\documentclass[11pt]{amsart}
\usepackage[T1]{fontenc}
\usepackage[utf8]{inputenc}
\usepackage{lmodern}
\usepackage{amsmath,amssymb,amsthm,mathtools}
\usepackage{microtype}
\usepackage{booktabs,array}
\usepackage{enumitem}
\usepackage{needspace}
\usepackage[hidelinks,unicode]{hyperref}
\usepackage{bookmark}

\numberwithin{equation}{section}
\newtheorem{theorem}{Theorem}
\newtheorem{lemma}[theorem]{Lemma}
\newtheorem{proposition}[theorem]{Proposition}
\newtheorem{corollary}[theorem]{Corollary}
\newtheorem{conjecture}[theorem]{Conjecture}
\theoremstyle{remark}
\newtheorem{remark}[theorem]{Remark}

\newcommand{\CM}{\mathrm{CM}}
\newcommand{\BF}{\mathrm{BF}}
\newcommand{\astar}{\alpha_*}
\newcommand{\dd}{\,\mathrm{d}}
\newcommand{\R}{\mathbb R}
\newcommand{\C}{\mathbb C}
\DeclareMathOperator{\Log}{Log}
\newcommand{\doi}[1]{\href{https://doi.org/#1}{\texttt{doi:#1}}}
\newcommand{\proofstep}[1]{\par\medskip\noindent\emph{#1}\ }
\allowdisplaybreaks[1]

\makeatletter
\def\@settitle{\begin{center}%
  \baselineskip14\p@\relax
  \bfseries
  \@title
  \end{center}%
}
\makeatother

\title[On the Alzer--Berg problem]{On the Alzer--Berg problem: an optimal Bernstein boundary and a uniqueness conjecture}
\author{Valmir Krasniqi}
\address{Institute of Science, Technology, Engineering and Mathematics (STEM), Rr. Nekibe Kelmendi Nr. 26, 10000 Prishtin\"e, Republic of Kosovo}
\email{valmir@stem-ks.org}
\date{}
\hypersetup{
  pdftitle={On the Alzer--Berg problem: an optimal Bernstein boundary and a uniqueness conjecture},
  pdfauthor={Valmir Krasniqi},
  pdfsubject={Bernstein functions and the two-parameter Alzer--Berg problem},
  pdfkeywords={Bernstein function, complete monotonicity, Laplace density, optimal boundary}
}

\begin{document}
\begin{abstract}
We study the two-parameter exponential family associated with the complete-monotonicity problem of Alzer and Berg.
Strict necessary parameter bounds place every possible Bernstein function in the domain of a regularized Laplace representation.
A quantitative positivity-transfer inequality then yields a linear sufficient condition with the largest possible universal coefficient, expressed in terms of the exact one-parameter critical exponent.
To describe the whole admissible region, we derive convolution identities for parameter derivatives and prove that increasing either normalized parameter destroys positivity at every zero of a nonnegative density.
Combined with uniform tail estimates, this excludes gaps in the admissible parameter intervals and gives a continuous, strictly monotone optimal boundary.
Global nonnegativity of the density and contact with zero characterize that boundary; its inverse determines the complete admissible interval for the second parameter.
The characterization is implicit and requires neither uniqueness nor nondegeneracy of the contact points.
Published numerical approximations of the one-parameter exponent are distinguished from the exact results and from the finite rational certificate used in the proofs.
We also derive a variational formula and a rigorous framework for validated numerical enclosure of the boundary, and formulate a boundary-contact conjecture asserting uniqueness and quadratic contact at every interior boundary point.
\end{abstract}
\maketitle

\section{Introduction}\label{sec:intro}
\subsection{Historical background and the Alzer--Berg problem}
Complete monotonicity relates alternating derivative inequalities to positive integral representations.
Its classical foundations include Bernstein's work on absolutely monotone functions \cite{Bernstein1929} and Widder's treatment of Laplace transforms \cite{Widder1941}.
The corresponding class of Bernstein functions is central to potential theory and the theory of convolution semigroups; systematic accounts are given by Berg and Forst \cite{BergForst1975}, Berg \cite{Berg2008}, and Schilling, Song and Vondra\v cek \cite{SSV2012}.
These connections make parameter-dependent elementary functions useful test cases: an elementary formula may conceal a delicate global positivity problem for its representing measure.

We consider
\begin{equation}\label{eq:family}
F_{a,b}(x)=\left(1+\dfrac ax\right)^{x+b},
\qquad a>0,\quad b\in\R,\quad x>0,
\end{equation}
and write $P_{a,b}(x)=e^a-F_{a,b}(x)$.
In 2002, Alzer and Berg \cite[p.~458]{AlzerBerg2002} asked for all pairs $(a,b)$ for which $P_{a,b}$ is completely monotone.
They established the case $b=0$, $0<a\le1$.
The same paper characterizes complete monotonicity of the opposite difference $F_{a,b}-e^a$ by the condition $a\le2b$ \cite[Theorem~4]{AlzerBerg2002}.
That result concerns a different sign and a different parameter region; it does not answer the question for $P_{a,b}$.
Their subsequent paper \cite{AlzerBerg2006} develops further classes of completely monotone functions and provides additional context for this line of investigation.

\subsection{The one-parameter problem and later developments}
Berg \cite{Berg2005} reformulated the slice $b=0$ in terms of the Bernstein property of
\begin{equation}\label{eq:hintro}
h_\alpha(y)=\left(1+\dfrac1y\right)^{\alpha y},
\qquad \alpha>0,\quad y>0.
\end{equation}
The scaling $F_{a,0}(ay)=h_a(y)$ identifies the two formulations.
Closure under positive powers of order at most one and under pointwise limits shows that the admissible exponents form an initial interval with an included endpoint; see also \cite[Section~1]{BMP2021}.
The problem therefore leads to an exact critical exponent $\astar$, rather than to a separate test for every derivative order.

Shemyakova, Khashin and Jeffrey \cite{SKJ2010} investigated the one-parameter problem computationally through high-order derivatives and obtained numerical evidence for a critical value near $2.299656443$.
Berg, Massa and Peron \cite{BMP2021} subsequently constructed entire Laplace densities for $e^\alpha-h_\alpha$ and derived coefficient expansions in terms of Bell polynomials.
Their work connects these densities with the Bessel function $J_1$ and the Lambert $W$ function, and reports the more precise approximation
\begin{equation}\label{eq:astar-decimal-intro}
\astar\approx2.29965644325346130332.
\end{equation}
The decimal in \eqref{eq:astar-decimal-intro} is taken from the part of \cite[Theorem~1.7]{BMP2021} explicitly designated as numerical results.
Here it is used only for illustration, not as an exact value or a certified enclosure.

Related work addresses stronger or complementary properties of the same one-parameter family.
Berg and Pedersen \cite{BergPedersen2023} identify those members whose derivatives are logarithmically completely monotone, leading to the class of Horn--Bernstein functions.
Cao, Guo, Du and Qi \cite{CGDQ2023} derive formulas for higher derivatives using partial Bell polynomials and give a one-parameter threshold characterization.
These results concern the slice \eqref{eq:hintro}; they do not supply the coordinatewise ordering and the two-parameter boundary established below.
In particular, the existence of the one-parameter threshold and its published numerical approximation are not claimed as new results of this paper.

\subsection{Main results and scope}
Recall that $g\in C^\infty(0,\infty)$ is \emph{completely monotone} if $(-1)^ng^{(n)}\ge0$ for every integer $n\ge0$.
A nonnegative smooth function $f$ is a \emph{Bernstein function} if $f'$ is completely monotone.
We write $\CM$ and $\BF$ for these classes.
The Bernstein representation is
\begin{equation}\label{eq:BFrep}
f(x)=k+dx+\int_{(0,\infty)}(1-e^{-xt})\,\mu(\mathrm{d}t),
\end{equation}
where $k,d\ge0$ and $\int_{(0,\infty)}\min\{1,t\}\,\mu(\mathrm{d}t)<\infty$; see \cite[Theorem~3.2]{SSV2012}.
Scaling the argument by a positive number preserves the Bernstein property.
Since $F_{a,b}(x)\to e^a$ as $x\to\infty$, we have
\begin{equation}\label{eq:equivalence-intro}
F_{a,b}\in\BF\quad\Longleftrightarrow\quad P_{a,b}\in\CM.
\end{equation}
Indeed, the derivative signs coincide, and a Bernstein $F_{a,b}$ is increasing and bounded above by its limit.
This is a special instance of the relation between bounded Bernstein functions and completely monotone complements; compare \cite[Remark~5.5]{Berg2008}.

The present paper separates three questions that must not be conflated.
Necessary inequalities restrict where the Bernstein property can hold.
An explicit sufficient region gives parameter pairs for which it does hold.
An exact boundary description must additionally account for every admissible pair, exclude gaps, and establish the behavior of the boundary as a parameter varies.
A sufficient line with the best universal coefficient need not coincide with that boundary at any fixed negative value of $b$.

Our first contribution is a regularized, integrable Laplace density whose domain covers every possible Bernstein pair.
The coverage follows from necessary inequalities proved independently of the positivity arguments.
A quantitative comparison between normalized densities then gives the sufficient condition $a\le\astar(1+b)$ for $-1<b\le0$.
The coefficient $\astar$ is optimal among constants valid for all such $b$, although the true admissible interval is strictly larger for each fixed $-1<b<0$.

The second contribution is a density-level argument for coordinatewise ordering of the full admissible set.
Convolution identities show that, at every zero of a nonnegative density, increasing either normalized parameter makes the density negative.
Uniform tail estimates turn this local fact into a global no-gap theorem.
Consequently, the whole Bernstein region is described by $0<a\le A(b)$, where $A$ is continuous and strictly increasing on $(-1,0]$.
The inverse boundary gives the complete admissible interval for $b$ when $a$ is fixed.

The resulting description is exact but implicit.
At the boundary the density is globally nonnegative and touches zero at a positive argument.
The global inequality is essential: solving the two tangency equations alone does not certify admissibility.
We neither assume nor prove uniqueness of the contact point, multiplicity exactly two, or differentiability of the boundary.
A variational formula and a rigorous framework for validated numerical enclosure are developed later.
Motivated by the contact equations and the local continuation theorem, we finally formulate the conjecture that every interior boundary density has exactly one positive zero and that this zero has multiplicity two.

\subsection{Organization and notation}
We use $\CM$ and $\BF$ for the classes of completely monotone and Bernstein functions, respectively. Throughout the normalized formulation we write $c=-b$, so that $0\le c<1$ corresponds to $-1<b\le0$. The symbol $\astar$ denotes the exact one-parameter critical exponent defined in Section~\ref{sec:transfer}; decimal values quoted for it are used only for numerical orientation.

The auxiliary results precede the main theorems.
Section~\ref{sec:necessary} proves the necessary restrictions, and Section~\ref{sec:density} constructs the density and its uniform tail estimates.
Section~\ref{sec:transfer} treats the exact one-parameter exponent and proves the positivity-transfer inequality.
The sharp linear sufficient theorem appears in Section~\ref{sec:linear}.
Section~\ref{sec:parameters} establishes the parameter-convolution identities and excludes gaps.
Section~\ref{sec:boundary} proves the optimal-boundary theorem, develops the variational and continuation formulations, and gives the inverse description.
Section~\ref{sec:comparisons} compares the bounds and records illustrative examples.
The paper closes in Section~\ref{sec:conjecture} with a boundary-contact uniqueness conjecture and several concrete open problems suggested by the theory.

\section{Necessary parameter bounds}\label{sec:necessary}
\subsection{The range of the second parameter}
\begin{lemma}[Necessary parameter restriction]\label{lem:b-range}
Let $a>0$ and $b\in\R$.
If $P_{a,b}\in\CM$, then
\begin{equation}\label{eq:b-range}
-1<b\le0.
\end{equation}
The same conclusion holds whenever $F_{a,b}\in\BF$.
\end{lemma}
\begin{proof}
If $P_{a,b}\in\CM$, then $F'_{a,b}=-P'_{a,b}\ge0$ and $F''_{a,b}=-P''_{a,b}\le0$.
These derivative inequalities also hold directly when $F_{a,b}\in\BF$.
Thus, in either case, $F_{a,b}$ is nondecreasing and concave.
As $x\downarrow0$, we have $F_{a,b}(x)\sim a^b x^{-b}$.
If $b>0$, this would tend to $+\infty$, which is incompatible with a nondecreasing function that is finite at every positive argument.
Hence $b\le0$.

It remains to exclude $b\le-1$.
Write $b=-c$, where $c\ge1$, and first suppose $c>1$.
As $x\downarrow0$,
\[
\log\left(1+\dfrac ax\right)=\log\dfrac ax+O(x),
\]
and therefore
\[
F_{a,-c}(x)=a^{-c}x^c\left(1+O\left(x\log\dfrac1x\right)\right).
\]
To justify the second-derivative asymptotics, set $g_c(x)=\log F_{a,-c}(x)$ and use the exact identities
\begin{equation}\label{eq:g-derivatives}
\begin{aligned}
g'_c(x)&=\log\left(1+\dfrac ax\right)+\dfrac cx-\dfrac{a+c}{x+a},\\
g''_c(x)&=-\dfrac a{x(x+a)}-\dfrac c{x^2}+\dfrac{a+c}{(x+a)^2}.
\end{aligned}
\end{equation}
In particular, $g'_c(x)=c/x+O(\log(1/x))$ and $g''_c(x)=-c/x^2+O(1/x)$.
Consequently,
\[
\begin{aligned}
F''_{a,-c}(x)&=F_{a,-c}(x)\bigl((g'_c(x))^2+g''_c(x)\bigr)\\
&=a^{-c}c(c-1)x^{c-2}(1+o(1)).
\end{aligned}
\]
Since $c(c-1)>0$, the second derivative is positive for all sufficiently small $x>0$, contradicting concavity.
Thus $b<-1$ is impossible.

For the borderline case $b=-1$, put $L(x)=\log(a/x)$.
Using
\[
\log\left(1+\dfrac ax\right)=L(x)+\dfrac xa+O(x^2),
\]
we obtain
\[
(x-1)\log\left(1+\dfrac ax\right)
=-L(x)+x\left(L(x)-\dfrac1a\right)+O(x^2).
\]
Exponentiating gives
\[
F_{a,-1}(x)=\dfrac xa+\dfrac{x^2}a\left(L(x)-\dfrac1a\right)
+O\bigl(x^3L(x)^2\bigr).
\]
Substitution of $c=1$ in \eqref{eq:g-derivatives}, together with this expansion and the identity $F''_{a,-1}=F_{a,-1}((g'_1)^2+g''_1)$, yields
\[
F''_{a,-1}(x)=\dfrac2a\log\dfrac ax-\dfrac3a-\dfrac2{a^2}
+O\bigl(xL(x)^2\bigr).
\]
Since $xL(x)^2\to0$, we conclude that $F''_{a,-1}(x)\to+\infty$ as $x\downarrow0$.
This again contradicts concavity, so $b=-1$ is impossible.
\end{proof}

For the admissible range $-1<b\le0$, put
\begin{equation}\label{eq:kappab}
\kappa_b=\begin{cases}0,&-1<b<0,\\1,&b=0.\end{cases}
\end{equation}
Then $F_{a,b}(0+)=\kappa_b$ and $F_{a,b}(\infty)=e^a$.
In every Bernstein representation of this family the linear term is zero and the L\'evy measure is finite, with mass $e^a-\kappa_b$.
This follows from monotone convergence in \eqref{eq:BFrep} and boundedness of $F_{a,b}$.

\subsection{Necessary upper bounds for the first parameter}
The Bernstein assumption also imposes quantitative restrictions on $a$.
The next proposition is independent of the positivity argument used to prove the main sufficient theorem.

\begin{proposition}[Necessary upper bounds]\label{prop:necessary}
Let $a>0$ and $b\in\R$.
If $F_{a,b}\in\BF$, or equivalently $P_{a,b}\in\CM$, then $-1<b\le0$ and
\begin{equation}\label{eq:necessary-bounds}
0<a<\min\left\{K_0+b,\;2(1+b)+2\sqrt{1+b}\right\},
\end{equation}
where
\begin{equation}\label{eq:K0}
K_0=\dfrac{30\log3-28}{(3\log3-2)^2}
\approx2.952828006549551.
\end{equation}
In particular,
\begin{equation}\label{eq:necessary-domain}
a-b<K_0<3.
\end{equation}
The displayed upper bounds are necessary conditions; no optimality is asserted for them.
\end{proposition}
\begin{proof}
The equivalence of the two assumptions follows from \eqref{eq:equivalence-intro}, and Lemma~\ref{lem:b-range} gives $-1<b\le0$.
Put $c=-b\in[0,1)$ and normalize the argument by writing
\begin{equation}\label{eq:H-necessary}
H(y)=F_{a,b}(ay)=\left(1+\dfrac1y\right)^{ay-c}.
\end{equation}
Then $H\in\BF$.

\proofstep{Logarithmic convexity and strictness.}
Two standard consequences of the Bernstein representation are $H'\in\CM$ and $H(y)/y\in\CM$; see \cite{Berg2008,SSV2012}.
For the second assertion, divide \eqref{eq:BFrep} by $y$ and use
\[
\dfrac{1-e^{-yt}}y=\int_0^t e^{-ys}\dd s.
\]
This gives a Laplace transform of a positive measure, including the nonnegative constant term.
Tonelli's theorem justifies the interchange of integrals.

If $g>0$ is completely monotone, write $g(y)=\int_{[0,\infty)}e^{-yt}\,\nu(\mathrm{d}t)$.
The probability measure $\nu_y(\mathrm{d}t)=e^{-yt}\nu(\mathrm{d}t)/g(y)$ satisfies
\begin{equation}\label{eq:variance}
(\log g(y))''=\int_{[0,\infty)}t^2\,\nu_y(\mathrm{d}t)
-\left(\int_{[0,\infty)}t\,\nu_y(\mathrm{d}t)\right)^2\ge0.
\end{equation}
Equality at even one positive $y$ forces $\nu$ to be supported at a single point.
Thus equality is possible only for $g(y)=Ce^{-\lambda y}$, with $C>0$ and $\lambda\ge0$.
For the function in \eqref{eq:H-necessary}, expansion at infinity gives
\begin{equation}\label{eq:H-infinity}
\dfrac{H(y)}y\sim\dfrac{e^a}y,
\qquad
H'(y)=\dfrac{e^a(a/2+c)}{y^2}+O(y^{-3}).
\end{equation}
In particular, $H'$ is not identically zero; complete monotonicity then implies $H'>0$.
Neither function in \eqref{eq:H-infinity} is a single exponential.
Both logarithmic-convexity inequalities are therefore strict for every $y>0$.

\proofstep{The bound from the quotient.}
Set $\delta=1-c=1+b>0$.
Direct differentiation gives
\begin{equation}\label{eq:quotient-log}
\left(\log\dfrac{H(y)}y\right)''
=\dfrac{y^2+(2\delta-a)y+\delta}{y^2(y+1)^2}>0.
\end{equation}
After division of the numerator by $y$, this says $a<2\delta+y+\delta/y$ for every $y>0$.
Taking $y=\sqrt\delta$ proves
\begin{equation}\label{eq:sqrt-bound}
a<2(1+b)+2\sqrt{1+b}.
\end{equation}
The strictness follows from \eqref{eq:variance}, not merely from minimization.

\proofstep{The bound from the derivative.}
Let $\ell(y)=\log H(y)$ and $d=3\log3-2$.
Since $H'=H\ell'>0$, strict logarithmic convexity of $H'$ gives
\begin{equation}\label{eq:log-derivative-test}
\ell'\ell'''-(\ell'')^2+(\ell')^2\ell''>0.
\end{equation}
At $y=1/2$, direct differentiation yields
\begin{equation}\label{eq:half-derivatives}
\ell'\left(\dfrac12\right)=\dfrac{da+4c}3,
\quad
\ell''\left(\dfrac12\right)=-\dfrac89(a+4c),
\quad
\ell'''\left(\dfrac12\right)=\dfrac{16}{27}(5a+26c).
\end{equation}
The left-hand side of \eqref{eq:log-derivative-test} at $1/2$ equals $8M(a,c)/81$, where
\begin{equation}\label{eq:Mpoly}
\begin{aligned}
M(a,c)={}&2(da+4c)(5a+26c)-8(a+4c)^2\\
&-(da+4c)^2(a+4c).
\end{aligned}
\end{equation}
Consequently $M(a,c)>0$.
Write $k=a+c=a-b$.
Since $K_0=(10d-8)/d^2$, expansion of \eqref{eq:Mpoly} at $a=k-c$ gives
\begin{equation}\label{eq:Mexpanded}
\begin{aligned}
M(k-c,c)={}&k^2(10d-8-d^2k)\\
&+k\bigl[32d-8-d(d+8)k\bigr]c\\
&+\bigl[(d-4)(5d+4)k+96-42d\bigr]c^2\\
&-3(d-4)^2c^3.
\end{aligned}
\end{equation}
We claim that $k\ge K_0$ makes this expression nonpositive.
The elementary bounds
\begin{equation}\label{eq:d-bounds}
\dfrac54<d<\dfrac43
\end{equation}
follow without decimal approximations from
\[
\log3=\sum_{n=0}^{\infty}\dfrac1{4^n(2n+1)},
\qquad
\dfrac{13}{12}<\log3<\dfrac{13}{12}+\dfrac15\sum_{n=2}^{\infty}4^{-n}
=\dfrac{11}{10}.
\]
The first term in \eqref{eq:Mexpanded} is nonpositive for $k\ge K_0$.
The coefficient in square brackets multiplying $kc$ decreases with $k$, and at $k=K_0$ it equals
\begin{equation}\label{eq:linear-poly-bound}
32d-8-d(d+8)K_0=\dfrac{2(11d^2-40d+32)}d<0.
\end{equation}
Indeed, $11d^2-40d+32$ decreases on $(5/4,4/3)$ and has value $-13/16$ at $5/4$.
The coefficient of $c^2$ also decreases with $k$, since $(d-4)(5d+4)<0$.
At $k=K_0$, it equals
\begin{equation}\label{eq:quadratic-poly-bound}
(d-4)(5d+4)K_0+96-42d
=\dfrac{8(d-1)(d^2-12d-16)}{d^2}<0.
\end{equation}
Finally, the coefficient of $c^3$ is negative.
Since $c\ge0$, all terms in \eqref{eq:Mexpanded} are nonpositive when $k\ge K_0$.
This contradicts $M(a,c)>0$, so $k<K_0$.
It remains to verify $K_0<3$.
By \eqref{eq:d-bounds},
\[
3d^2-10d+8=(3d-4)(d-2)>0,
\]
which is equivalent to $(10d-8)/d^2<3$.
Together with \eqref{eq:sqrt-bound}, this proves all the assertions.
\end{proof}

For later comparison, denote the necessary upper envelope by
\begin{equation}\label{eq:U}
U(b)=\min\left\{K_0+b,\;2(1+b)+2\sqrt{1+b}\right\},
\qquad -1<b\le0.
\end{equation}
In particular, along any sequence of Bernstein parameters with $b\downarrow-1$, the first parameter must satisfy $a\to0$.
Also, \eqref{eq:necessary-domain} shows that every possible Bernstein pair lies in the domain $a-b<3$ of the regularized representation in Section~\ref{sec:density}.
This is a necessary-domain statement, not a positivity assertion throughout that domain.

\section{The regularized Laplace density}\label{sec:density}
\subsection{Normalization and the Bernstein criterion}
It is convenient to replace $b$ by $c=-b$ and work with
\begin{equation}\label{eq:H}
H_{a,c}(y)=\left(1+\dfrac1y\right)^{ay-c},
\qquad a>0,\quad 0\le c<1.
\end{equation}
Then
\begin{equation}\label{eq:scaling}
F_{a,b}(ay)=H_{a,-b}(y).
\end{equation}
Multiplying the argument by a positive constant preserves the Bernstein property, so the two parameterizations are equivalent.
Write $\kappa(c)=1$ for $c=0$ and $\kappa(c)=0$ for $0<c<1$.
The endpoint limits are
\begin{equation}\label{eq:H-endpoints}
H_{a,c}(0+)=\kappa(c),\qquad H_{a,c}(\infty)=e^a.
\end{equation}
Indeed, $y\log(1+1/y)\to0$ at the origin, whereas $y\log(1+1/y)\to1$ at infinity.
It follows that
\begin{equation}\label{eq:H-equivalence}
H_{a,c}\in\BF\quad\Longleftrightarrow\quad e^a-H_{a,c}\in\CM.
\end{equation}
For the forward implication, boundedness and monotonicity give $H_{a,c}\le e^a$; all higher derivative inequalities are identical on the two sides.
The reverse implication follows from these derivative inequalities and positivity of $H_{a,c}$.

\subsection{The regularized representation}
The one-parameter entire densities in \cite{BMP2021} motivate a density formulation of the two-parameter problem.
Here the endpoint singularity is handled by two explicit regularizing subtractions.
Define
\begin{equation}\label{eq:weight}
w_{a,c}(s)=\dfrac1\pi\left(\dfrac{s}{1-s}\right)^{as+c}
\sin\bigl(\pi(as+c)\bigr),\qquad 0<s<1,
\end{equation}
and set
\begin{equation}\label{eq:q12}
q_1=\dfrac a2+c,
\qquad
q_2=\dfrac a3+\dfrac c2+\dfrac12\left(\dfrac a2+c\right)^2.
\end{equation}
The dependence of these coefficients on $(a,c)$ is understood.

\begin{proposition}[Regularized Laplace representation]\label{prop:density}
Suppose $a>0$, $0\le c<1$ and $a+c<3$.
Define
\begin{equation}\label{eq:D}
\begin{aligned}
D_{a,c}(t)={}&q_1+(q_1-q_2)t\\
&+e^{-a}\int_0^1 w_{a,c}(s)
\bigl[e^{(1-s)t}-1-(1-s)t\bigr]\dd s,
\end{aligned}
\end{equation}
and
\begin{equation}\label{eq:Phi}
\Phi_{a,c}(t)=e^{a-t}D_{a,c}(t).
\end{equation}
Both functions are entire in $t$, and $\Phi_{a,c}\in L^1(0,\infty)$.
For $y>0$,
\begin{equation}\label{eq:Laplace}
e^a-H_{a,c}(y)=\int_0^\infty e^{-yt}\Phi_{a,c}(t)\dd t.
\end{equation}
Consequently,
\begin{equation}\label{eq:density-criterion}
H_{a,c}\in\BF\quad\Longleftrightarrow\quad D_{a,c}(t)\ge0
\ \text{for every }t\ge0.
\end{equation}
For every integer $n\ge2$,
\begin{equation}\label{eq:D-derivatives}
D^{(n)}_{a,c}(t)=e^{-a}\int_0^1(1-s)^n w_{a,c}(s)e^{(1-s)t}\dd s.
\end{equation}
\end{proposition}
\begin{proof}
Extend $H_{a,c}$ analytically to $\C\setminus[-1,0]$ by using the principal logarithm in \eqref{eq:H}.
On the upper bank of the cut,
\[
\Log\left(1+\dfrac1{-s+i0}\right)=\log\left(\dfrac{1-s}s\right)-i\pi.
\]
Therefore
\begin{equation}\label{eq:cut}
\begin{aligned}
H_{a,c}(-s+i0)&=\left(\dfrac{s}{1-s}\right)^{as+c}e^{i\pi(as+c)},\\
\operatorname{Im}H_{a,c}(-s+i0)&=\pi w_{a,c}(s).
\end{aligned}
\end{equation}
Put $P(z)=e^a-H_{a,c}(z)$.
Expansion at infinity gives
\[
P(z)=e^a\left(\dfrac{q_1}z-\dfrac{q_2}{z^2}+O(z^{-3})\right).
\]
Hence
\[
Q(z)=(z+1)^2P(z)-e^a[q_1z+2q_1-q_2]
\]
is $O(z^{-1})$ at infinity.

We spell out the contour step, since the regularization at the two endpoints is essential.
Fix $y>0$, choose $R>y+2$, and let $\Gamma_{R,\varepsilon}$ be the positively oriented boundary of the disk $|z|<R$ cut along $[-1,0]$, with the endpoint disks $|z+1|<\varepsilon$ and $|z|<\varepsilon$ removed.
Thus $\Gamma_{R,\varepsilon}$ consists of the outer circle, the two small endpoint circles, and the upper and lower banks of the segment $[-1+\varepsilon,-\varepsilon]$.
Cauchy's formula gives
\[
Q(y)=\dfrac{1}{2\pi i}\int_{\Gamma_{R,\varepsilon}}\dfrac{Q(z)}{z-y}\dd z.
\]
Since $Q(z)=O(z^{-1})$, the integrand on $|z|=R$ is $O(R^{-2})$ and the outer-circle contribution is $O(R^{-1})$, hence tends to zero as $R\to\infty$.

Near $z=-1$,
\[
H_{a,c}(z)=O(|z+1|^{-a-c}),
\qquad Q(z)=O(1+|z+1|^{2-a-c}).
\]
Because $y>0$, the factor $(z-y)^{-1}$ is bounded on the small circle about $-1$; its contribution to Cauchy's formula is therefore
\[
O(\varepsilon)+O(\varepsilon^{3-a-c})\longrightarrow0
\]
under the hypothesis $a+c<3$.
At $z=0$, $H_{a,c}$ and hence $Q$ are bounded for $c\ge0$, so the small-circle contribution is $O(\varepsilon)$ and again vanishes.

Write $Q_\pm(-s)$ for the boundary values on the upper and lower banks.
For real $a,c$, the lower boundary value is the complex conjugate of the upper one, so \eqref{eq:cut} gives
\[
H_{a,c}(-s+i0)-H_{a,c}(-s-i0)=2\pi i\,w_{a,c}(s).
\]
Since the polynomial subtracted in the definition of $Q$ has no jump,
\[
\begin{aligned}
Q_+(-s)-Q_-(-s)
&=-(1-s)^2\bigl(H_{a,c}(-s+i0)-H_{a,c}(-s-i0)\bigr)\\
&=-2\pi i(1-s)^2w_{a,c}(s).
\end{aligned}
\]
The slit is an inner boundary, hence it is traversed clockwise: the upper bank runs from $-1$ to $0$ and the lower bank from $0$ to $-1$.
Consequently the two bank integrals combine as
\[
\dfrac{1}{2\pi i}\int_{\varepsilon}^{1-\varepsilon}
\dfrac{Q_-(-s)-Q_+(-s)}{y+s}\dd s.
\]
Letting first $R\to\infty$ and then $\varepsilon\downarrow0$, and using the jump identity above, yields
\[
Q(y)=\int_0^1\dfrac{(1-s)^2w_{a,c}(s)}{y+s}\dd s.
\]
Consequently,
\begin{equation}\label{eq:regularized-stieltjes}
\begin{aligned}
P(y)={}&e^a\left[\dfrac{q_1}{y+1}+\dfrac{q_1-q_2}{(y+1)^2}\right]\\
&+\int_0^1\dfrac{(1-s)^2w_{a,c}(s)}{(y+s)(y+1)^2}\dd s.
\end{aligned}
\end{equation}
The exponential remainder in \eqref{eq:D} is $O((1-s)^2)$, locally uniformly for complex $t$.
Since $w_{a,c}(s)=O((1-s)^{-a-c})$ as $s\uparrow1$, the condition $a+c<3$ guarantees convergence at that endpoint and local uniformity in $t$.
At $s=0$, the weight is $O(s^c)$ for $c>0$ and $O(s)$ for $c=0$.
This proves the asserted entire dependence and justifies \eqref{eq:D-derivatives}.

For real $t\ge0$, the remainder is nonnegative, and
\begin{equation}\label{eq:remainder-mass}
\int_0^\infty e^{-t}\bigl[e^{(1-s)t}-1-(1-s)t\bigr]\dd t
=\dfrac{(1-s)^2}s.
\end{equation}
The integral of $|w_{a,c}(s)|(1-s)^2/s$ over $(0,1)$ is finite.
Its behavior at zero is $O(s^{c-1})$ when $c>0$, and it is bounded when $c=0$.
Its behavior at one is $O((1-s)^{2-a-c})$.
Tonelli's theorem applied to the absolute value therefore proves $\Phi_{a,c}\in L^1$.
Finally,
\[
\begin{aligned}
&\int_0^\infty e^{-(y+1)t}\bigl[e^{(1-s)t}-1-(1-s)t\bigr]\dd t\\
&\hspace{1em}=\dfrac1{y+s}-\dfrac1{y+1}-\dfrac{1-s}{(y+1)^2}
=\dfrac{(1-s)^2}{(y+s)(y+1)^2}.
\end{aligned}
\]
Fubini's theorem and \eqref{eq:regularized-stieltjes} prove \eqref{eq:Laplace}.
Dominated convergence at $y=0$ gives
\begin{equation}\label{eq:Phi-mass}
\int_0^\infty\Phi_{a,c}(t)\dd t=e^a-\kappa(c).
\end{equation}
If $D_{a,c}\ge0$, subtracting \eqref{eq:Laplace} from \eqref{eq:Phi-mass} gives a Bernstein representation.
Conversely, \eqref{eq:H-equivalence} and uniqueness of Laplace transforms identify the signed density in \eqref{eq:Laplace} with a positive finite measure; see \cite{Widder1941,SSV2012}.
Continuity then gives pointwise nonnegativity.
\end{proof}

\begin{remark}[Coverage of all possible Bernstein parameters]\label{rem:coverage}
By Proposition~\ref{prop:necessary} and \eqref{eq:scaling}, $H_{a,c}\in\BF$ entails $a+c<K_0<3$.
Thus the hypothesis $a+c<3$ in Proposition~\ref{prop:density} excludes no possible Bernstein function in this family.
For $a+c\ge3$, the Bernstein property is already ruled out by a necessary parameter bound, independently of the contour argument.
Conversely, $a+c<3$ alone guarantees only a signed Laplace representation, not positivity of its density.
\end{remark}

\subsection{Coefficient identities}
For $k\ge1$, put
\begin{equation}\label{eq:ellk}
\ell_k=\dfrac a{k+1}+\dfrac ck.
\end{equation}
Define $q_0(a,c)=1$ and, for $n\ge1$,
\begin{equation}\label{eq:recurrence}
q_n(a,c)=\dfrac1n\sum_{k=1}^n\left(\dfrac{ak}{k+1}+c\right)q_{n-k}(a,c).
\end{equation}
Indeed, the generating function
\[
\exp\left(\sum_{k=1}^\infty\ell_kz^k\right)
=\sum_{n=0}^\infty q_n(a,c)z^n,\qquad |z|<1,
\]
gives \eqref{eq:recurrence} upon differentiation.
The first four coefficients are
\begin{equation}\label{eq:q1234}
\begin{aligned}
q_1&=\ell_1,&q_2&=\ell_2+\dfrac{\ell_1^2}2,\\
q_3&=\ell_3+\ell_1\ell_2+\dfrac{\ell_1^3}6,&
q_4&=\ell_4+\ell_1\ell_3+\dfrac{\ell_2^2}2
+\dfrac{\ell_1^2\ell_2}2+\dfrac{\ell_1^4}{24}.
\end{aligned}
\end{equation}
Expansion of $H_{a,c}$ at infinity yields
\[
e^{-a}(e^a-H_{a,c}(y))
=\sum_{n=1}^\infty(-1)^{n+1}q_n(a,c)y^{-n},\qquad |y|>1.
\]
Comparison with the Laplace expansion of \eqref{eq:Laplace} gives
\begin{equation}\label{eq:D0}
D^{(n)}_{a,c}(0)=\sum_{k=0}^n(-1)^k\binom nk q_{k+1}(a,c),
\qquad n\ge0.
\end{equation}
For completeness, expand $\Phi_{a,c}$ at zero in its Laplace integral.
The integral away from zero is exponentially small as $y\to\infty$, and the integral of each Taylor monomial is $n!/y^{n+1}$.
Since $\Phi_{a,c}=e^{a-t}D_{a,c}$, the binomial relation \eqref{eq:D0} follows.
These identities extend the coefficient description of the one-parameter density in \cite{BMP2021}; related Bell-polynomial formulas for derivatives of $h_\alpha$ are given in \cite{CGDQ2023}.

\subsection{Uniform behavior at large arguments}
\begin{lemma}[Uniform positive tails]\label{lem:tails}
Fix $c\in[0,1)$ and a compact interval $I\subset(0,3-c)$.
Uniformly for $a\in I$, as $t\to\infty$,
\begin{equation}\label{eq:tails}
\Phi_{a,c}(t)=
\begin{cases}
\dfrac a{t^2}+O\left(\dfrac{\log t}{t^3}\right),&c=0,\\[9pt]
\dfrac{\Gamma(c+1)\sin(\pi c)}{\pi t^{c+1}}
+O\left(\dfrac{\log t}{t^{c+2}}\right),&0<c<1.
\end{cases}
\end{equation}
For $c>0$, the second expansion is also uniform when $(a,c)$ ranges over an arbitrary compact subset of
\begin{equation}\label{eq:interior-domain}
\{(a,c):a>0,\ 0<c<1,\ a+c<3\}.
\end{equation}
More quantitatively, if $I\subset(0,3)$ is compact, there exist constants $C_{I,0}>0$ and $T_{I,0}\ge2$ such that
\begin{equation}\label{eq:tail-bound-c0}
\left|\Phi_{a,0}(t)-\dfrac a{t^2}\right|
\le C_{I,0}\dfrac{1+\log t}{t^3},
\qquad a\in I,\quad t\ge T_{I,0}.
\end{equation}
If $K$ is a compact subset of \eqref{eq:interior-domain}, then there exist constants $C_K>0$ and $T_K\ge2$ such that
\begin{equation}\label{eq:tail-bound-K}
\left|\Phi_{a,c}(t)-\dfrac{\Gamma(c+1)\sin(\pi c)}{\pi t^{c+1}}\right|
\le C_K\dfrac{1+\log t}{t^{c+2}},
\qquad (a,c)\in K,\quad t\ge T_K.
\end{equation}
The constants need not be evaluated numerically for the qualitative arguments below.
In either form, the densities are positive for all sufficiently large $t$, uniformly on the indicated parameter sets.
No joint uniformity at $c=0$ is asserted.
\end{lemma}
\begin{proof}
Fix $\eta\in(0,1)$ and split the integral in \eqref{eq:D} at $s=\eta$.
On $[\eta,1)$, the estimate
\[
|e^{(1-s)t}-1-(1-s)t|\le\dfrac12(1-s)^2t^2e^{(1-s)t}
\]
shows that its contribution to $\Phi_{a,c}(t)$ is $O(t^2e^{-\eta t})$, uniformly on each of the parameter sets under consideration.
The polynomial term in \eqref{eq:D}, and the subtraction terms over $(0,\eta)$, are $O((1+t)e^{-t})$ uniformly on the same sets.
Thus the only nonexponentially small contribution is
\[
\int_0^\eta e^{-st}w_{a,c}(s)\dd s.
\]

For $c=0$, the identities
\[
\left(\dfrac{s}{1-s}\right)^{as}=1+O_I(s|\log s|),
\qquad
\sin(\pi as)=\pi as+O_I(s^3)
\]
show that
\[
w_{a,0}(s)=as+O_I(s^2(1+|\log s|))
\qquad(s\downarrow0)
\]
uniformly for $a\in I$.
Choose $\eta$ small enough and define
\[
R_{I,0}:=
\sup_{\substack{a\in I\\0<s\le\eta}}
\dfrac{|w_{a,0}(s)-as|}{s^2(1+|\log s|)}<\infty.
\]
Then, for $t\ge2$,
\[
\begin{aligned}
\left|\int_0^\eta e^{-st}\bigl(w_{a,0}(s)-as\bigr)\dd s\right|
&\le \dfrac{R_{I,0}}{t^3}
\int_0^{\eta t}e^{-u}u^2\bigl(1+|\log u-\log t|\bigr)\dd u\\
&\le \dfrac{R_{I,0}(A_0+B_0\log t)}{t^3},
\end{aligned}
\]
where
\[
A_0:=\int_0^\infty e^{-u}u^2(1+|\log u|)\dd u,
\qquad B_0:=\Gamma(3).
\]
Replacing the truncated integral of $as$ by $a\int_0^\infty e^{-st}s\dd s=a/t^2$ introduces only an exponentially small error.
Enlarging the constant to absorb all exponentially small terms gives \eqref{eq:tail-bound-c0}.

Now let $K$ be compact in \eqref{eq:interior-domain}, and put
\[
\lambda(c):=\dfrac{\sin(\pi c)}\pi.
\]
Since $K$ stays a positive distance from $c=0$, $c=1$, and $a+c=3$, the elementary expansions
\[
\begin{aligned}
\left(\dfrac{s}{1-s}\right)^{as+c}
&=s^c\bigl[1+O_K(s(1+|\log s|))\bigr],\\
\sin(\pi(as+c))&=\sin(\pi c)+O_K(s).
\end{aligned}
\]
give
\[
w_{a,c}(s)=\lambda(c)s^c+O_K(s^{c+1}(1+|\log s|))
\qquad(s\downarrow0)
\]
uniformly on $K$.
Hence one may choose the same $\eta\in(0,1)$ for all $(a,c)\in K$ and define
\begin{equation}\label{eq:RK}
R_K:=
\sup_{\substack{(a,c)\in K\\0<s\le\eta}}
\dfrac{|w_{a,c}(s)-\lambda(c)s^c|}
{s^{c+1}(1+|\log s|)}<\infty.
\end{equation}
Let $[c_-,c_+]$ be the projection of $K$ onto the $c$-axis and set
\[
A_K:=\sup_{c\in[c_-,c_+]}
\int_0^\infty e^{-u}u^{c+1}(1+|\log u|)\dd u,
\qquad
B_K:=\sup_{c\in[c_-,c_+]}\Gamma(c+2).
\]
Both constants are finite.
With $u=st$, \eqref{eq:RK} gives, uniformly on $K$,
\begin{equation}\label{eq:local-tail-error}
\left|\int_0^\eta e^{-st}\bigl(w_{a,c}(s)-\lambda(c)s^c\bigr)\dd s\right|
\le R_K\dfrac{A_K+B_K\log t}{t^{c+2}}.
\end{equation}
The difference between
$\lambda(c)\int_0^\eta e^{-st}s^c\dd s$
and
$\lambda(c)\Gamma(c+1)t^{-c-1}$
is exponentially small uniformly for $c\in[c_-,c_+]$; the same is true of the already separated terms from $[\eta,1)$ and of the polynomial/subtraction terms.
Since an exponential is dominated uniformly by any fixed inverse power for $t\ge2$, these terms can be absorbed by increasing a single constant $C_K$.
This proves \eqref{eq:tail-bound-K}, and hence \eqref{eq:tails}.

Finally, the leading coefficients are uniformly bounded away from zero on the relevant compact parameter sets.
For example, with
\[
L_K:=\min_{c\in[c_-,c_+]}
\dfrac{\Gamma(c+1)\sin(\pi c)}\pi>0,
\]
any $t\ge T_K$ satisfying $C_K(1+\log t)/t\le L_K/2$ obeys
\[
\Phi_{a,c}(t)\ge\dfrac{L_K}{2t^{c+1}}>0
\qquad((a,c)\in K).
\]
The case $c=0$ is identical with $L_I:=\min I>0$ and \eqref{eq:tail-bound-c0}.
This proves the stated uniform positivity of the tails.
\end{proof}

\section{The one-parameter threshold and positivity transfer}\label{sec:transfer}
\subsection{An exact localization of the one-parameter threshold}
Define
\begin{equation}\label{eq:astar}
h_\alpha(y)=\left(1+\dfrac1y\right)^{\alpha y},
\qquad
\astar=\sup\{\alpha>0:h_\alpha\in\BF\}.
\end{equation}
The parameter-interval description is discussed in \cite{Berg2005,BMP2021,CGDQ2023}.
For the transfer proof we need an exact localization of the endpoint, so a self-contained argument is included.
Every theorem below uses the exact constant in \eqref{eq:astar}, not the decimal approximation \eqref{eq:astar-decimal-intro}.

\begin{proposition}[Exact one-parameter threshold]\label{prop:threshold}
The constant in \eqref{eq:astar} is well defined, satisfies
\begin{equation}\label{eq:astar-localization}
2<\astar<\dfrac73,
\end{equation}
and has the characterization
\begin{equation}\label{eq:alpha-interval}
h_\alpha\in\BF\quad\Longleftrightarrow\quad0<\alpha\le\astar.
\end{equation}
\end{proposition}
\begin{proof}
\proofstep{Positivity at $\alpha=2$.}
Write $D=D_{2,0}$ and $w=w_{2,0}$.
From \eqref{eq:D0},
\[
D(0)=1,\qquad D'(0)=-\dfrac16,\qquad D''(0)=0,\qquad
D'''(0)=\dfrac1{360}.
\]
The weight satisfies
\[
w(1-s)=-\left(\dfrac{1-s}s\right)^2w(s),
\qquad w(s)>0\quad(0<s<1/2).
\]
Pairing $s$ and $1-s$ in \eqref{eq:D-derivatives} gives
\[
D''(t)=e^{-2}\int_0^{1/2}(1-s)^2w(s)
\bigl[e^{(1-s)t}-e^{st}\bigr]\dd s.
\]
Because $2\sinh v\ge2v$ for $v\ge0$,
\[
e^{(1-s)t}-e^{st}\ge(1-2s)te^{t/2}.
\]
Hence $D''(t)\ge te^{t/2}/360$.
Integrating twice from zero yields
\begin{equation}\label{eq:G}
D(t)\ge G(t/2),\qquad
G(z)=1-\dfrac z3+\dfrac{(z-2)e^z+z+2}{45}.
\end{equation}
To check positivity, let $J(z)=(z-2)e^z+z+2$.
Then $J(0)=J'(0)=0$ and $J''(z)=ze^z$, so $J(z)>0$ for $z>0$.
For $0\le z\le3$, this gives $G(z)>0$, including the endpoint $z=3$.
For $z\ge3$,
\[
G'(z)=-\dfrac13+\dfrac{(z-1)e^z+1}{45}>0.
\]
Thus $D_{2,0}(t)>0$ for every $t\ge0$, proving $h_2\in\BF$.

\proofstep{Strictness of the lower bound.}
Lemma~\ref{lem:tails}, with $c=0$ and $a$ in a compact neighborhood of $2$, gives uniform positivity for $t\ge T$, with some finite $T$.
On $[0,T]$, the continuous function $\Phi_{2,0}$ has a positive minimum.
Formula \eqref{eq:D} is jointly continuous in $(a,t)$ on a compact neighborhood of this set.
Therefore $\Phi_{a,0}>0$ on $[0,T]$ for all $a>2$ sufficiently close to $2$.
Together with the uniform tail estimate, this proves that some $h_a$ with $a>2$ belongs to $\BF$.

\proofstep{An exact upper bound.}
Set $p_n=q_n(7/3,0)$.
The entire Taylor expansion implied by \eqref{eq:D0} is
\begin{equation}\label{eq:Phi-series}
e^{-\alpha}\Phi_{\alpha,0}(t)
=\sum_{n=0}^\infty(-1)^n q_{n+1}(\alpha,0)\dfrac{t^n}{n!},
\qquad 0<\alpha<3.
\end{equation}
We verify that the alternating tail can be bounded at $\alpha=7/3$, $t=5$.
For $\alpha\ge2$, the recurrence \eqref{eq:recurrence}, with $c=0$, implies
\[
(n+1)(q_{n+1}-q_n)
=\left(\dfrac\alpha2-1\right)q_n
+\alpha\sum_{k=1}^n\dfrac{q_{n-k}}{(k+1)(k+2)}\ge0.
\]
Here every $q_n$ is positive.
Consequently,
\[
\dfrac{q_{n+1}}{q_n}
=\dfrac\alpha{n+1}\sum_{k=1}^{n+1}\dfrac k{k+1}\dfrac{q_{n+1-k}}{q_n}
\le\alpha.
\]
Thus the absolute terms $p_{n+1}5^n/n!$ decrease for $n\ge11$ and tend to zero.
For the even partial sum
\begin{equation}\label{eq:S20}
S_{20}=\sum_{n=0}^{20}(-1)^n p_{n+1}\dfrac{5^n}{n!},
\end{equation}
exact rational arithmetic gives $S_{20}=-M/N$, where
\begin{equation}\label{eq:certificate}
\begin{aligned}
M&=6318802906638319646617316077208715843189283,\\
N&=31980502025532614891359112401173930231082254336.
\end{aligned}
\end{equation}
In particular, $6000M>N$ and $S_{20}<-1/6000$.
The alternating-series bound therefore gives
\[
e^{-7/3}\Phi_{7/3,0}(5)\le S_{20}<-\dfrac1{6000}<0.
\]
Thus $h_{7/3}\notin\BF$.
This is a finite exact certificate, not a numerical estimate of $\astar$.

\proofstep{The parameter interval and its endpoint.}
If $h_\alpha\in\BF$ and $0<\beta<\alpha$, then $h_\beta=h_\alpha^{\beta/\alpha}\in\BF$, since positive powers of order at most one preserve the Bernstein class \cite{Berg2008,SSV2012}.
Hence the admissible set is an initial interval.
It contains parameters larger than $2$ and cannot contain any parameter at least $7/3$.
To see that its supremum is included, take admissible $\alpha_j\to\astar$.
For every fixed $y>0$ and every derivative order,
\[
h_{\alpha_j}^{(n)}(y)\longrightarrow h_{\astar}^{(n)}(y).
\]
All Bernstein sign inequalities pass to the limit.
Thus $h_{\astar}\in\BF$.
Since $h_{7/3}\notin\BF$, the upper inequality in \eqref{eq:astar-localization} is strict.
\end{proof}

\begin{remark}\label{rem:exact-localization}
The exact localization \eqref{eq:astar-localization} is sufficient for the transfer argument.
The published decimal \eqref{eq:astar-decimal-intro} is used only to illustrate the size of the sufficient region.
Neither a certified enclosure for all its displayed digits nor a description of the critical zeros is required.
\end{remark}

\subsection{A scalar sine inequality}
Throughout the remainder of this section assume
\begin{equation}\label{eq:transfer-range}
2\le\alpha\le\dfrac73,\qquad0<\delta<1.
\end{equation}
For brevity, write $D_\delta=D_{\alpha\delta,1-\delta}$ and $D_1=D_{\alpha,0}$.
These densities are covered by Proposition~\ref{prop:density}, because $\alpha\delta+1-\delta=1+(\alpha-1)\delta<3$.

\begin{lemma}\label{lem:sine}
Let $0<\theta<\pi$, $Q>0$ and $0<\delta<1$.
If $Q\le\theta/\sin\theta$, then
\begin{equation}\label{eq:sine}
Q^\delta\sin(\delta\theta)\ge\delta Q\sin\theta.
\end{equation}
\end{lemma}
\begin{proof}
The function $L(v)=\log(\sin(v\theta)/(v\theta))$ extends continuously to $L(0)=0$.
For $0<v\le1$,
\[
L''(v)=\dfrac1{v^2}-\dfrac{\theta^2}{\sin^2(v\theta)}<0.
\]
Concavity gives $L(\delta)\ge\delta L(1)$, and hence
\[
\sin(\delta\theta)\ge\delta\sin\theta
\left(\dfrac\theta{\sin\theta}\right)^{1-\delta}
\ge\delta\sin\theta\,Q^{1-\delta}.
\]
Multiplication by $Q^\delta$ proves the assertion.
\end{proof}

\subsection{A single sign change of the difference kernel}
Define
\begin{equation}\label{eq:W}
W_\delta(s)=e^{-\alpha\delta}w_{\alpha\delta,1-\delta}(s)
-\delta e^{-\alpha}w_{\alpha,0}(s),\qquad0<s<1.
\end{equation}
\begin{lemma}\label{lem:sign-change}
Under \eqref{eq:transfer-range},
\begin{equation}\label{eq:W-sign}
W_\delta(s)\ge0\quad(0<s<1/\alpha),
\qquad
W_\delta(s)\le0\quad(1/\alpha<s<1).
\end{equation}
\end{lemma}
\begin{proof}
Put
\[
r=\dfrac{s}{1-s},\qquad u=\alpha s-1,\qquad Q=e^{-\alpha}r^u.
\]
Substitution in \eqref{eq:W} gives the exact identity
\begin{equation}\label{eq:W-exact}
\pi W_\delta(s)=r\bigl[\delta Q\sin(\pi u)-Q^\delta\sin(\pi\delta u)\bigr].
\end{equation}

\proofstep{The range $-1<u<1$.}
For $-1<u<0$, let $v=1+u\in(0,1)$.
Then
\[
(1+u)Q=e^{-\alpha}(\alpha-v)^{1-v}v^v
\le\alpha e^{-\alpha}\le2e^{-2}.
\]
For $0<u<1$, the function
\[
\alpha\longmapsto e^{-\alpha}\left(\dfrac{1+u}{\alpha-1-u}\right)^u
\]
is decreasing for $\alpha\ge2$.
Therefore
\[
(1-u)Q\le e^{-2}(1+u)^u(1-u)^{1-u}\le2e^{-2}.
\]
In both cases,
\begin{equation}\label{eq:Q-bound}
(1-|u|)Q\le2e^{-2}<\dfrac12.
\end{equation}
For $0<v<1$, the elementary bound $\sin(\pi v)\le\pi\min\{v,1-v\}$ gives
\[
(1-v)\dfrac{\pi v}{\sin(\pi v)}\ge\dfrac12.
\]
Together with \eqref{eq:Q-bound}, this permits the application of Lemma~\ref{lem:sine} with $\theta=\pi|u|$.
If $u<0$, the right-hand side of \eqref{eq:W-exact} is nonnegative.
If $u>0$, it is nonpositive.
At $u=0$ it vanishes.

\proofstep{The range $1\le u<\alpha-1$.}
The endpoint $u=1$, when present, follows directly from \eqref{eq:W-exact}.
For $u>1$, the assumptions imply $u<\alpha-1\le4/3$ and
\[
r=\dfrac{1+u}{\alpha-1-u}\ge6,
\qquad Q\ge6e^{-7/3}>e^{-1}.
\]
If $\delta u\le1$, then $\sin(\pi u)<0$ and $\sin(\pi\delta u)\ge0$, so \eqref{eq:W-exact} is nonpositive.
Suppose instead that $\delta u>1$, and set $\theta=\pi u$.
Consider
\[
K(v)=\dfrac{Q^v\sin(v\theta)}v,
\qquad \dfrac\pi\theta<v\le1.
\]
Writing $z=v\theta$, the sign of $K'(v)$ is the sign of
\[
\sin z\,[v\log Q-1+z\cot z].
\]
Here $\pi<z\le4\pi/3$, so $\sin z<0$.
Moreover, $v\log Q>-1$, and $z\cot z$ is decreasing on this interval, with
\[
z\cot z\ge\dfrac{4\pi}{3\sqrt3}>2.
\]
Thus $K'(v)<0$ and $K(\delta)\ge K(1)$.
This is equivalent to $Q^\delta\sin(\delta\theta)\ge\delta Q\sin\theta$, which again makes \eqref{eq:W-exact} nonpositive.
All possible values of $u$ have now been covered.
\end{proof}

\subsection{The quantitative transfer estimate}
\begin{theorem}[Positivity transfer]\label{thm:transfer}
Under \eqref{eq:transfer-range}, for every $t\ge0$,
\begin{equation}\label{eq:transfer}
D_\delta(t)\ge\delta D_1(t)+(1-\delta)\left(1-\dfrac89\delta\right).
\end{equation}
In particular, nonnegativity of $D_{\alpha,0}$ implies strict positivity of $D_{\alpha\delta,1-\delta}$.
\end{theorem}
\begin{proof}
Define
\begin{equation}\label{eq:R}
R_\delta(t)=\dfrac{D_\delta(t)-\delta D_1(t)-(1-\delta)}{\delta(1-\delta)}
\end{equation}
and put $\varepsilon=\alpha-2\in[0,1/3]$.
For the parameters of $D_\delta$, \eqref{eq:ellk} becomes
\begin{equation}\label{eq:ell-transfer}
\ell_k=\dfrac1k+\delta\left(\dfrac{2+\varepsilon}{k+1}-\dfrac1k\right).
\end{equation}
Substitution of \eqref{eq:ell-transfer} into the four coefficients in \eqref{eq:q1234}, followed by \eqref{eq:D0}, gives
\begin{align}
R_\delta(0)&=0,\qquad R'_\delta(0)=\dfrac{\varepsilon^2}8,\label{eq:R01}\\
R''_\delta(0)&=-\dfrac{\varepsilon[(1+\delta)\varepsilon^2+2\varepsilon+4]}{48},\label{eq:R2}\\
1152R'''_\delta(0)&=3(\delta^2+\delta+1)\varepsilon^4
+(24\delta-8)\varepsilon^2-32\varepsilon+16.\label{eq:R3}
\end{align}
These expressions use only $q_1,q_2,q_3,q_4$.
In particular, the relevant combinations are $D'(0)=q_1-q_2$, $D''(0)=q_1-2q_2+q_3$, and $D'''(0)=q_1-3q_2+3q_3-q_4$.
Introduce the rational constants
\begin{equation}\label{eq:BC}
B=\dfrac{11}{324},\qquad C=\dfrac{121}{31104}.
\end{equation}
Equation \eqref{eq:R2} implies $R''_\delta(0)\ge-B$ by using $\delta\le1$ and $\varepsilon\le1/3$.
After deleting the nonnegative terms involving $\delta$ in \eqref{eq:R3}, we obtain
\[
1152R'''_\delta(0)\ge3\varepsilon^4-8\varepsilon^2-32\varepsilon+16.
\]
The polynomial on the right is decreasing on $[0,1/3]$, because its derivative is $12\varepsilon^3-16\varepsilon-32<0$ there.
Its value at $1/3$ is $121/27$.
Hence
\begin{equation}\label{eq:R-bounds}
R''_\delta(0)\ge-B,\qquad R'''_\delta(0)\ge C>0.
\end{equation}
Set $\lambda=1-1/\alpha\ge1/2$.
By \eqref{eq:D-derivatives} and \eqref{eq:W},
\[
R^{(4)}_\delta(t)=\dfrac1{\delta(1-\delta)}
\int_0^1(1-s)^4W_\delta(s)e^{(1-s)t}\dd s.
\]
To compare this with $R'''_\delta(0)$, observe that $(1-s)^3W_\delta(s)$ is nonnegative for $s<1/\alpha$ and nonpositive for $s>1/\alpha$.
The factor $(1-s)e^{(1-s)t}$ is at least $\lambda e^{\lambda t}$ on the first interval and at most this value on the second.
Thus Lemma~\ref{lem:sign-change} gives
\begin{equation}\label{eq:R4}
R^{(4)}_\delta(t)\ge\lambda e^{\lambda t}R'''_\delta(0)\ge\dfrac C2,
\qquad t\ge0.
\end{equation}
All these integrals converge absolutely by Proposition~\ref{prop:density}.
Taylor's formula with integral remainder, using \eqref{eq:R01}, \eqref{eq:R-bounds} and \eqref{eq:R4}, now yields
\[
\begin{aligned}
R_\delta(t)
&=R'_\delta(0)t+\dfrac{R''_\delta(0)}2t^2+\dfrac{R'''_\delta(0)}6t^3
+\dfrac16\int_0^t(t-v)^3R^{(4)}_\delta(v)\dd v\\
&\ge-\dfrac B2t^2+\dfrac C{48}t^4\\
&=\dfrac C{48}\left(t^2-\dfrac{12B}C\right)^2-\dfrac{3B^2}C
\ge-\dfrac89.
\end{aligned}
\]
The last equality uses $3B^2/C=8/9$.
Substituting this inequality into \eqref{eq:R} proves \eqref{eq:transfer}.
Finally, $(1-\delta)(1-8\delta/9)>0$ for $0<\delta<1$.
\end{proof}

\section{The sharp linear region}\label{sec:linear}
The preceding auxiliary results yield the explicit sufficient condition.
Its optimality concerns a universal coefficient, not the entire two-parameter boundary.

\begin{theorem}[Sharp linear sufficient condition]\label{thm:linear}
Let $a>0$ and $-1<b\le0$.
If
\begin{equation}\label{eq:linear-condition}
a\le\astar(1+b),
\end{equation}
then $F_{a,b}\in\BF$ and
\begin{equation}\label{eq:Frep}
F_{a,b}(x)=\kappa_b+\int_0^\infty(1-e^{-xt})\rho_{a,b}(t)\dd t,
\qquad x>0,
\end{equation}
where $\rho_{a,b}$ is a continuous, nonnegative, integrable density, explicitly given in \eqref{eq:rho} below.
Its total mass is
\begin{equation}\label{eq:rho-mass}
\int_0^\infty\rho_{a,b}(t)\dd t=e^a-\kappa_b.
\end{equation}
If $b<0$, then, with $\tau=a/\astar-b\in(-b,1]$,
\begin{equation}\label{eq:rho-lower}
\rho_{a,b}(t)\ge\dfrac{-ab}9\left(1-\dfrac{8b}\tau\right)e^{a-at}>0,
\qquad t\ge0.
\end{equation}
Moreover, $\astar$ is the largest constant $C$ for which
\[
0<a\le C(1+b)\quad\Longrightarrow\quad F_{a,b}\in\BF
\]
holds for every $-1<b\le0$.
\end{theorem}
\begin{proof}
\proofstep{Positivity on the comparison line.}
Take $\alpha=\astar$ in Theorem~\ref{thm:transfer}; this is permitted by Proposition~\ref{prop:threshold}.
Since $h_{\astar}\in\BF$, Proposition~\ref{prop:density} gives $D_{\astar,0}(t)\ge0$ for $t\ge0$.
For $0<c<1$, choose $\delta=1-c$.
The transfer inequality becomes
\begin{equation}\label{eq:line-positivity}
D_{\astar(1-c),c}(t)
\ge(1-c)D_{\astar,0}(t)+\dfrac{c(1+8c)}9
\ge\dfrac{c(1+8c)}9>0.
\end{equation}
Therefore $H_{\astar(1-c),c}\in\BF$.
The endpoint $c=0$ is exactly the definition of $\astar$.
This proves positivity on the entire comparison line.

\proofstep{Extension by a positive convolution expansion.}
Suppose $0\le c<1$ and $0<a\le\astar(1-c)$.
Set
\begin{equation}\label{eq:tau}
\tau=\dfrac a\astar+c,\qquad A_0=\dfrac a\tau,\qquad C_0=\dfrac c\tau.
\end{equation}
Then $0<\tau\le1$, $0\le C_0<1$, and
\[
A_0=\astar(1-C_0),\qquad H_{a,c}(y)=H_{A_0,C_0}(y)^\tau.
\]
Thus the usual closure property under powers already proves $H_{a,c}\in\BF$.
We give the density-level argument as well, both to identify the measure and to retain a quantitative lower bound.
Define
\[
m(t)=e^{-A_0}\Phi_{A_0,C_0}(t),
\qquad \widehat m(y)=\int_0^\infty e^{-yt}m(t)\dd t.
\]
The comparison-line result gives $m\ge0$, and \eqref{eq:Phi-mass} shows $\|m\|_1\le1$.
Moreover,
\[
H_{A_0,C_0}(y)=e^{A_0}(1-\widehat m(y)).
\]
For $0<\tau<1$, let
\begin{equation}\label{eq:binomial}
\beta_n=(-1)^{n+1}\binom\tau n
=\dfrac\tau{n!}\prod_{j=1}^{n-1}(j-\tau)>0,
\end{equation}
where $\beta_1=\tau$.
For $\tau=1$, set $\beta_1=1$ and $\beta_n=0$ for $n\ge2$.
The binomial expansion is
\[
1-(1-z)^\tau=\sum_{n=1}^\infty\beta_nz^n,
\qquad0\le z<1.
\]
Because $\beta_n\ge0$, the partial sums increase monotonically as $z\uparrow1$; hence monotone convergence gives
\[
\sum_{n=1}^\infty\beta_n
=\lim_{z\uparrow1}\sum_{n=1}^\infty\beta_nz^n
=\lim_{z\uparrow1}\bigl[1-(1-z)^\tau\bigr]=1.
\]
For $\tau=1$ this is immediate from the convention above.
If $m^{*n}$ denotes the $n$-fold convolution on $[0,\infty)$, then
\begin{equation}\label{eq:positive-convolution}
\Phi_{a,c}=e^a\sum_{n=1}^\infty\beta_nm^{*n}
\end{equation}
as an identity in $L^1(0,\infty)$.
Indeed, the right-hand side converges in $L^1$ because $\sum\beta_n\|m\|_1^n\le1$, and its Laplace transform is $e^a-H_{a,c}$.
Uniqueness and Proposition~\ref{prop:density} identify it with $\Phi_{a,c}$.
The proposition applies here since
\[
a+c\le\astar(1-c)+c\le\astar<3.
\]
Every summand in \eqref{eq:positive-convolution} is nonnegative.
In particular,
\[
\Phi_{a,c}(t)\ge\tau e^{a-A_0}\Phi_{A_0,C_0}(t),
\qquad D_{a,c}(t)\ge\tau D_{A_0,C_0}(t).
\]
These inequalities hold almost everywhere from \eqref{eq:positive-convolution}, and everywhere by continuity.
If $c>0$, \eqref{eq:line-positivity} therefore yields
\begin{equation}\label{eq:D-lower}
D_{a,c}(t)\ge\dfrac c9\left(1+\dfrac{8c}\tau\right)>0.
\end{equation}

\proofstep{Return to the original variables.}
By \eqref{eq:scaling} and a change of variables in \eqref{eq:Laplace}, the original density is
\begin{equation}\label{eq:rho}
\rho_{a,b}(t)=a\Phi_{a,-b}(at)=ae^{a-at}D_{a,-b}(at).
\end{equation}
Equations \eqref{eq:D}, \eqref{eq:weight} and \eqref{eq:q12} make this expression fully explicit.
Written directly in the original parameters, it is
\begin{equation}\label{eq:rho-original}
\begin{aligned}
\rho_{a,b}(t)=e^{-at}\biggl\{&e^a[p_1+(ap_1-p_2)t]\\
&+\int_0^a\omega_{a,b}(s)
\bigl[e^{(a-s)t}-1-(a-s)t\bigr]\dd s\biggr\},
\end{aligned}
\end{equation}
where
\begin{equation}\label{eq:p12}
p_1=\dfrac{a^2}2-ab,
\qquad p_2=\dfrac{a^3}3-\dfrac{ba^2}2+\dfrac{p_1^2}2,
\end{equation}
and
\begin{equation}\label{eq:omega}
\omega_{a,b}(s)=\dfrac1\pi\left(\dfrac{s}{a-s}\right)^{s-b}
\sin\bigl(\pi(s-b)\bigr),\qquad0<s<a.
\end{equation}
The integration limits and the two regularizing subtractions in \eqref{eq:rho-original} are essential.
The weight $\omega_{a,b}$ need not itself be nonnegative.
It is the complete density $\rho_{a,b}$ whose nonnegativity has been established.

The mass identity \eqref{eq:rho-mass} follows from \eqref{eq:Phi-mass}.
Subtracting the Laplace representation from this mass identity proves \eqref{eq:Frep}.
For $b<0$, \eqref{eq:D-lower} becomes exactly \eqref{eq:rho-lower}.
For every integer $n\ge1$, differentiation under the integral gives
\begin{equation}\label{eq:strict-derivatives}
(-1)^{n-1}F_{a,b}^{(n)}(x)
=\int_0^\infty t^ne^{-xt}\rho_{a,b}(t)\dd t>0,
\qquad x>0.
\end{equation}
The differentiations are justified by $\rho_{a,b}\in L^1$ and boundedness of $t^ne^{-xt}$ away from $x=0$.
Strictness follows because the nonnegative density is continuous and $\rho_{a,b}(0)=e^ap_1>0$.

\proofstep{Optimality of the universal coefficient.}
Suppose a constant $C>\astar$ satisfied the implication in the theorem for every $-1<b\le0$.
Taking $b=0$ and $a=C$ would give $F_{C,0}\in\BF$.
But $F_{C,0}(Cy)=h_C(y)$, contradicting \eqref{eq:alpha-interval}.
Thus no coefficient larger than $\astar$ is possible.
\end{proof}

\subsection{Why the sufficient line is not the exact boundary}
\begin{proposition}[Local extension beyond the sufficient line]\label{prop:extension}
For each fixed $-1<b<0$, there exists $\eta_b>0$ such that
\begin{equation}\label{eq:extension}
0<a\le\astar(1+b)+\eta_b\quad\Longrightarrow\quad F_{a,b}\in\BF.
\end{equation}
The proposition does not specify an explicit value of $\eta_b$.
\end{proposition}
\begin{proof}
Fix $c=-b\in(0,1)$ and let $a_0=\astar(1-c)$.
The comparison-line estimate \eqref{eq:line-positivity} gives $\Phi_{a_0,c}(t)>0$ for every $t\ge0$.
Choose a compact interval $I\subset(0,3-c)$ whose interior contains $a_0$.
Lemma~\ref{lem:tails} gives positivity of $\Phi_{a,c}(t)$ for all $t\ge T$ and all $a\in I$, with $T$ independent of $a$.
On $[0,T]$, $\Phi_{a_0,c}$ has a positive minimum.
Joint continuity, which follows from \eqref{eq:D} by dominated convergence on compact parameter sets, preserves this positivity when $a$ is sufficiently close to $a_0$.
Thus $H_{a,c}\in\BF$ for $a$ in some open interval around $a_0$.
Combining this interval with the region already covered by Theorem~\ref{thm:linear} proves \eqref{eq:extension}.
\end{proof}

This proposition makes the scope of sharpness precise.
The coefficient is sharp because of the endpoint $b=0$, although every fixed negative $b$ admits a local improvement beyond the line.
The exact implicit boundary established in Section~\ref{sec:boundary} lies above the sufficient comparison line at every fixed negative $b$.

\section{Parameter monotonicity and exclusion of gaps}\label{sec:parameters}
The linear estimate does not by itself imply that the admissible parameters have no gaps.
We now establish the additional property needed for an exact boundary description.
Throughout this section let
\begin{equation}\label{eq:Omega-m}
\begin{gathered}
\Omega=\{(a,c):a>0,\ 0\le c<1,\ a+c<3\},\\
m_{a,c}(t)=e^{-a}\Phi_{a,c}(t)=e^{-t}D_{a,c}(t).
\end{gathered}
\end{equation}
By Proposition~\ref{prop:density}, $H_{a,c}\in\BF$ if and only if $m_{a,c}\ge0$ on $[0,\infty)$.
Also,
\begin{equation}\label{eq:m0}
m_{a,c}(0)=\dfrac a2+c>0.
\end{equation}

\subsection{Convolution identities}
For locally integrable functions on $[0,\infty)$, write $(f*g)(t)=\int_0^t f(t-s)g(s)\dd s$.
Define
\begin{equation}\label{eq:kjQ}
\begin{aligned}
k(t)&=\int_0^1re^{-rt}\dd r=\dfrac{1-(1+t)e^{-t}}{t^2},\\
j(t)&=\int_0^1e^{-rt}\dd r=\dfrac{1-e^{-t}}t,\\
Q_{a,c}(t)&=t\bigl[ak(t)+cj(t)\bigr]=at k(t)+c(1-e^{-t}),\qquad t>0.
\end{aligned}
\end{equation}
The removable singularities are filled in by $k(0)=1/2$ and $j(0)=1$.

\begin{lemma}[Parameter and Volterra identities]\label{lem:convolution}
The density $m_{a,c}(t)$ is continuously differentiable in its parameters, locally uniformly for $t$ in compact intervals, with the $c$-derivative understood from the right at $c=0$.
For $(a,c)\in\Omega$ and $t\ge0$,
\begin{equation}\label{eq:parameter-identities}
\partial_a m_{a,c}=k-k*m_{a,c},
\qquad\partial_c m_{a,c}=j-j*m_{a,c},
\end{equation}
and
\begin{equation}\label{eq:Volterra}
tm_{a,c}(t)=Q_{a,c}(t)-(Q_{a,c}*m_{a,c})(t).
\end{equation}
\end{lemma}
\begin{proof}
Set $L(y)=\log(1+1/y)$ and $u(y)=1-yL(y)$.
The normalized Laplace representation is
\begin{equation}\label{eq:mhat}
\widehat m_{a,c}(y)=\int_0^\infty e^{-yt}m_{a,c}(t)\dd t
=1-\exp[-au(y)-cL(y)],\qquad y>0.
\end{equation}
Fubini's theorem in \eqref{eq:kjQ} gives
\[
\widehat k(y)=\int_0^1\dfrac r{y+r}\dd r=u(y),
\qquad
\widehat j(y)=\int_0^1\dfrac1{y+r}\dd r=L(y).
\]
To justify parameter differentiation without an endpoint difficulty at $c=0$, put $g_{a,c}=ak+cj$ and use the locally convergent series
\begin{equation}\label{eq:m-convolution-series}
m_{a,c}(t)=\sum_{n=1}^\infty\dfrac{(-1)^{n+1}}{n!}g_{a,c}^{*n}(t).
\end{equation}
Indeed, on $0\le t\le T$ and any compact parameter set, $|g_{a,c}(t)|\le M$ for a fixed $M$, and
\[
|g_{a,c}^{*n}(t)|\le M^n\dfrac{t^{n-1}}{(n-1)!}.
\]
The series and its parameter derivatives therefore converge locally uniformly.
It also converges absolutely in every weighted space $L^1(e^{-yt}\dd t)$, because
\[
\sum_{n=1}^\infty\dfrac{\|e^{-y\cdot}g_{a,c}\|_1^n}{n!}
=\exp[au(y)+cL(y)]-1<\infty.
\]
Its Laplace transform is the right-hand side of \eqref{eq:mhat}; uniqueness identifies it with the contour density.
The same local argument even gives analytic dependence on the parameters.
Differentiating \eqref{eq:mhat} now yields
\[
\partial_a\widehat m=u(1-\widehat m),\qquad
\partial_c\widehat m=L(1-\widehat m).
\]
These are the Laplace transforms of \eqref{eq:parameter-identities}.
Finally,
\[
-\partial_y\widehat m=(1-\widehat m)(-au'-cL'),
\qquad\widehat Q=-au'-cL',
\]
which proves \eqref{eq:Volterra}.
All convolution identities hold first in a weighted $L^1$ space and then pointwise by continuity.
No unweighted integrability of $\partial_c m$ at $c=0$ is used.
\end{proof}

\subsection{Strict parameter derivatives at every nonnegative zero}
\begin{lemma}[Direction of loss of positivity]\label{lem:loss}
Suppose $(a,c)\in\Omega$, $m_{a,c}(t)\ge0$ for all $t\ge0$, and $m_{a,c}(t_0)=0$ at some $t_0>0$.
Then
\begin{equation}\label{eq:negative-parameters}
\partial_a m_{a,c}(t_0)<0,\qquad\partial_c m_{a,c}(t_0)<0.
\end{equation}
\end{lemma}
\begin{proof}
Suppress the parameter subscripts.
The function $k$ is positive and strictly decreasing by its integral expression in \eqref{eq:kjQ}.
Consequently
\[
\dfrac{Q(t)}{k(t)}=at+c\dfrac{1-e^{-t}}{k(t)}.
\]
We justify the monotonicity of the second term explicitly. Since
\[
j(t)=\int_0^1 e^{-rt}\,\dd r,
\qquad
k(t)=\int_0^1 r e^{-rt}\,\dd r,
\]
the quotient $k(t)/j(t)$ is the expectation of $r$ under the probability measure
\[
\frac{e^{-rt}}{j(t)}\,\dd r
\qquad (0\le r\le1).
\]
Differentiation gives
\[
\left(\frac{k}{j}\right)'(t)
=-\operatorname{Var}_t(r)<0,
\]
because this probability measure is not concentrated at one point. Hence $j/k$ is strictly increasing. Since
\[
\frac{1-e^{-t}}{k(t)}=t\frac{j(t)}{k(t)},
\]
the second term above is nondecreasing (and strictly increasing when $c>0$); together with the strictly increasing term $at$, this proves that $Q/k$ is strictly increasing on $(0,\infty)$.
Moreover,
\begin{equation}\label{eq:Qj}
\dfrac{Q(t)}{j(t)}=a\left(1-\dfrac t{e^t-1}\right)+ct
\end{equation}
is strictly increasing, since
\[
\dfrac{\mathrm{d}}{\mathrm{d}t}\left(1-\dfrac t{e^t-1}\right)
=\dfrac{(t-1)e^t+1}{(e^t-1)^2}>0.
\]
The numerator is positive because it vanishes at zero and has derivative $te^t>0$.
Thus $p/Q$ is strictly decreasing for either choice $p=k$ or $p=j$.
At $t=t_0$, the Volterra identity gives
\begin{equation}\label{eq:Volterra-zero}
Q(t_0)=\int_0^{t_0}Q(t_0-s)m(s)\dd s.
\end{equation}
For $0<s<t_0$,
\[
p(t_0-s)>\dfrac{p(t_0)}{Q(t_0)}Q(t_0-s).
\]
Integration against $m(s)\dd s$ and \eqref{eq:Volterra-zero} give $(p*m)(t_0)>p(t_0)$.
Strictness follows from \eqref{eq:m0}: $m$ is positive on an interval next to zero.
The identities in Lemma~\ref{lem:convolution} now prove both inequalities in \eqref{eq:negative-parameters}.
\end{proof}

\subsection{Local inward perturbations and exclusion of gaps}
\begin{lemma}[Local positivity under a parameter decrease]\label{lem:inward}
Let $(a,c)\in\Omega$ and suppose $m_{a,c}\ge0$ on $[0,\infty)$.
For all sufficiently small $h>0$,
\begin{equation}\label{eq:a-decrease}
m_{a-h,c}(t)>0\qquad(t\ge0).
\end{equation}
If $c>0$, then also, for all sufficiently small $h>0$,
\begin{equation}\label{eq:c-decrease}
m_{a,c-h}(t)>0\qquad(t\ge0).
\end{equation}
If the original density is strictly positive everywhere, this positivity persists under sufficiently small changes of $a$ with $c$ fixed.
For $c>0$, it also persists under sufficiently small joint changes of $(a,c)$.
\end{lemma}
\begin{proof}
First make the positivity near the origin uniform in the parameters.
Choose a compact parameter neighborhood $N\subset\Omega$ containing all sufficiently small perturbations under consideration; when $c=0$, take $N$ one-sided in the $c$-variable.
By \eqref{eq:m0},
\[
\mu_0:=\min_{(r,d)\in N}\left(\dfrac r2+d\right)>0.
\]
Lemma~\ref{lem:convolution} gives joint continuity of $m_{r,d}(t)$ on compact parameter sets, locally uniformly in $t$.
After shrinking $N$ if necessary, there exists $\varepsilon>0$ such that
\[
m_{r,d}(t)\ge\dfrac{\mu_0}{2}>0,
\qquad (r,d)\in N,\quad0\le t\le\varepsilon.
\]
Thus positivity on a whole interval next to $t=0$, not merely at the single point $t=0$, is uniform under the permitted perturbations.

At the other end, the tail estimates of Lemma~\ref{lem:tails} are positive uniformly in the same parameter neighborhoods.
For changes of $a$ with $c=0$, use the fixed-$c$ estimate.
For changes of $c>0$, choose $N$ entirely inside $0<c<1$, where the joint estimate applies.
Increasing $T$ if necessary, these observations reduce every remaining sign question to the fixed compact interval $[\varepsilon,T]\subset(0,\infty)$.

The zeros of the original density in that interval are finite in number, because the density is entire and is not identically zero.
If there are no zeros, compactness and continuity prove the stated stability.
Otherwise, choose small neighborhoods of all zeros.
Lemma~\ref{lem:loss} and continuity of the parameter derivatives show that, on these neighborhoods and for parameters sufficiently close to the original ones, the relevant parameter derivative is at most $-\gamma$ for some $\gamma>0$.
For instance,
\[
m_{a-h,c}(t)=m_{a,c}(t)-\int_{a-h}^a\partial_r m_{r,c}(t)\dd r
\ge m_{a,c}(t)+\gamma h>0.
\]
The same argument applies to a decrease of $c$.
On the compact complement of the zero neighborhoods, the original density has a positive minimum, which persists by continuity.
Combining these estimates with the origin and tail estimates proves the lemma.
\end{proof}

\begin{proposition}[Coordinatewise decrease of parameters]\label{prop:coordinate}
If $a>0$, $0\le c<1$ and $H_{a,c}\in\BF$, then
\begin{equation}\label{eq:coordinate}
H_{a',c'}\in\BF\qquad\text{whenever}\qquad0<a'\le a,\quad0\le c'\le c.
\end{equation}
If at least one inequality is strict, then $m_{a',c'}(t)>0$ for every $t\ge0$.
\end{proposition}
\begin{proof}
The set of admissible parameters is relatively closed in $(0,\infty)\times[0,1)$: for each fixed $y>0$ and each derivative order, the derivatives of $H_{a,c}(y)$ depend continuously on $(a,c)$, and all Bernstein sign inequalities pass to the limit.
Every admissible pair satisfies $a+c<K_0<3$, so decreasing a coordinate stays inside the representation domain.

Fix one coordinate and vary the other.
Suppose a larger value $q$ is admissible but a smaller positive value $p$ is not.
The closed nonempty set of admissible values in $[p,q]$ has a minimum $r>p$.
Lemma~\ref{lem:inward} gives admissible values immediately to the left of $r$, contradicting its minimality.
For the $c$-coordinate this argument first applies to $0<p<q$; passage to $p=0$ is justified by relative closedness.
This proves \eqref{eq:coordinate} by decreasing the two coordinates successively.

For strict positivity, suppose a coordinate has been strictly decreased and the resulting admissible density has a zero.
By Lemma~\ref{lem:loss}, a sufficiently small increase of that coordinate makes the density negative at the zero.
But the increased parameter still lies below the original admissible coordinate, so \eqref{eq:coordinate} says that it must remain admissible.
This is a contradiction.
The argument also applies at $c'=0$, using the right derivative there; no tail estimate uniform across $c=0$ is needed.
\end{proof}

The content of Proposition~\ref{prop:coordinate} is specific to this family.
It is not an inference that arbitrary products, quotients or parameter changes preserve the Bernstein class.
In the original variables, increasing $b$ towards zero corresponds to decreasing $c$.

\section{The optimal boundary}\label{sec:boundary}
We now distinguish an optimal boundary from the explicit sufficient line.
The interval property is proved first, rather than assumed in the definition of a supremum.

\begin{theorem}[The exact boundary as a monotone graph]\label{thm:boundary}
There is a unique continuous, strictly increasing function
\[
A:(-1,0]\longrightarrow(0,\astar]
\]
such that, for all $a>0$ and $b\in\R$,
\begin{equation}\label{eq:boundary}
F_{a,b}\in\BF\quad\Longleftrightarrow\quad-1<b\le0
\quad\text{and}\quad0<a\le A(b).
\end{equation}
It satisfies
\begin{equation}\label{eq:A-endpoints}
A(0)=\astar,\qquad\lim_{b\downarrow-1}A(b)=0.
\end{equation}
For $-1<b<0$,
\begin{equation}\label{eq:A-bounds}
\astar(1+b)<A(b)<\min\{\astar,U(b)\}.
\end{equation}
For a fixed $b\in(-1,0]$, the density is strictly positive at every $t\ge0$ whenever $0<a<A(b)$.
At $a=A(b)$ it is nonnegative and has a nonempty finite set of zeros in $(0,\infty)$.
\end{theorem}
\begin{proof}
\proofstep{Existence of an attained threshold.}
For $0\le c<1$, let
\[
S_c=\{a>0:H_{a,c}\in\BF\}.
\]
Theorem~\ref{thm:linear} shows that $S_c$ is nonempty.
The necessary bounds show that it is bounded above, and Proposition~\ref{prop:coordinate} shows that it is an initial interval.
Define $T(c)=\sup S_c$.
Since $T(c)>0$, relative closedness gives $T(c)\in S_c$.
Consequently,
\begin{equation}\label{eq:T-attained}
S_c=(0,T(c)],\qquad T(c)+c<K_0<3.
\end{equation}
The exact one-parameter result gives $T(0)=\astar$.
Set $A(b)=T(-b)$.
Lemma~\ref{lem:b-range} and \eqref{eq:scaling} then prove \eqref{eq:boundary}.
The endpoint of a given interval is unique.

\proofstep{Contact with zero and strict positivity below the boundary.}
Proposition~\ref{prop:coordinate} gives strict positivity when $a<T(c)$.
At $a=T(c)$, if the density were strictly positive everywhere, Lemma~\ref{lem:inward} would preserve positivity for slightly larger $a$, a contradiction.
It must therefore have a zero.
Its positivity at the origin, eventual positivity, and analyticity show that the zero set in $[0,\infty)$ is finite and contained in $(0,\infty)$.

\proofstep{Strict monotonicity.}
Let $0\le c_1<c_2<1$ and take $a=T(c_2)$.
Decreasing $c_2$ to $c_1$ gives an everywhere positive density by Proposition~\ref{prop:coordinate}.
The fixed-$c_1$ stability in Lemma~\ref{lem:inward} permits a further small increase of $a$.
Thus $T(c_1)>T(c_2)$, proving that $A$ is strictly increasing.
In particular, $A(b)<\astar$ for $b<0$.

\proofstep{Continuity in the interior.}
Fix $c\in(0,1)$ and let $c_n\to c$.
The sufficient bound $T(c_n)\ge\astar(1-c_n)$ bounds these thresholds away from zero, and the necessary bound bounds them above.
Every convergent subsequence $T(c_{n_j})\to t$ has an admissible limiting pair $(t,c)$ by relative closedness.
Hence $t\le T(c)$, proving $\limsup T(c_n)\le T(c)$.
Conversely, fix $0<a<T(c)$.
Its density is strictly positive everywhere.
The joint tail estimate and compact-interval stability in Lemma~\ref{lem:inward} make $(a,c_n)$ admissible for all sufficiently large $n$.
Therefore $\liminf T(c_n)\ge a$.
Letting $a\uparrow T(c)$ proves continuity at $c$.

\proofstep{Endpoint continuity and quantitative envelopes.}
For $c\downarrow0$,
\[
\astar(1-c)\le T(c)\le T(0)=\astar,
\]
so $T(c)\to\astar$.
This argument avoids any assertion of a tail estimate uniform across $c=0$.
As $c\uparrow1$, the quotient bound gives
\[
0<T(c)<2(1-c)+2\sqrt{1-c}\longrightarrow0.
\]
The necessary bound, applied to the attained admissible pair $(T(c),c)$, gives $A(b)<U(b)$.
Proposition~\ref{prop:extension} gives $A(b)>\astar(1+b)$ for $b<0$.
All claims follow.
\end{proof}

\subsection{A necessary and sufficient zero-contact criterion}
\begin{corollary}[Global positivity and contact characterize the boundary]\label{cor:contact}
Let $a>0$, $-1<b\le0$ and $a-b<3$.
Then $a=A(b)$ if and only if
\begin{equation}\label{eq:contact}
D_{a,-b}(t)\ge0\quad\text{for all }t\ge0,
\qquad D_{a,-b}(t_b)=0\quad\text{for some }t_b>0.
\end{equation}
Every such zero has finite even order at least two; in particular,
\begin{equation}\label{eq:tangent}
\partial_tD_{a,-b}(t_b)=0,\qquad\partial_t^2D_{a,-b}(t_b)\ge0.
\end{equation}
At a contact point, the parameter derivatives satisfy
\begin{equation}\label{eq:D-param-contact}
\partial_aD_{a,c}(t_b)<0,\qquad\partial_cD_{a,c}(t_b)<0,
\qquad c=-b.
\end{equation}
\end{corollary}
\begin{proof}
The forward implication follows from Theorem~\ref{thm:boundary}.
Conversely, global nonnegativity in \eqref{eq:contact} gives $a\le A(b)$ by the Bernstein criterion.
The existence of a zero excludes $a<A(b)$ by the strict-positivity assertion of that theorem.
Thus $a=A(b)$.
A zero of a nonnegative, nonzero real-analytic function at an interior point has finite even order.
Finally, $D=e^tm$, so Lemma~\ref{lem:loss} gives \eqref{eq:D-param-contact}.
\end{proof}

\begin{remark}[Scope of the implicit characterization]\label{rem:implicit}
The system $D_{a,-b}(t)=\partial_tD_{a,-b}(t)=0$ provides candidate contact points, but it is not sufficient on its own.
The global inequality $D_{a,-b}\ge0$ in \eqref{eq:contact} is essential.
A tangency can occur while the density is negative elsewhere.
The theorem does not assert uniqueness of the contact point, exact multiplicity two, or differentiability of $A$.
Thus \eqref{eq:boundary} and \eqref{eq:contact} are an exact implicit description, not an elementary closed formula or a certified numerical evaluation of the boundary.
\end{remark}

\subsection{Variational characterization and numerical evaluation of the boundary}\label{subsec:numerical-boundary}
The zero-contact criterion also gives a direct numerical formulation of the optimal boundary.
For fixed $-1<b<0$, define
\begin{equation}\label{eq:Psi-def}
\Psi_b(a)=\inf_{t\ge0}D_{a,-b}(t),
\qquad 0<a<3+b.
\end{equation}
The restriction $a<3+b$ places the pair in the domain of Proposition~\ref{prop:density}.
For the interval used below this restriction is automatic, because every upper endpoint is smaller than the necessary envelope in Proposition~\ref{prop:necessary}.

\begin{proposition}[Variational formula for the optimal boundary]\label{prop:variational-A}
For every $-1<b<0$,
\begin{equation}\label{eq:A-variational}
A(b)=\sup\{a>0:\Psi_b(a)\ge0\}
=\sup\left\{a>0:\inf_{t\ge0}D_{a,-b}(t)\ge0\right\}.
\end{equation}
Moreover,
\begin{equation}\label{eq:Psi-sign}
\Psi_b(a)>0\quad(0<a<A(b)),\qquad
\Psi_b(A(b))=0,
\end{equation}
and, whenever $A(b)<a<3+b$,
\begin{equation}\label{eq:Psi-negative}
\Psi_b(a)<0.
\end{equation}
At $a=A(b)$ the infimum in \eqref{eq:Psi-def} is attained at at least one positive argument.
\end{proposition}
\begin{proof}
By Proposition~\ref{prop:density}, within the representation domain one has
$D_{a,-b}\ge0$ on $[0,\infty)$ if and only if $F_{a,b}\in\BF$.
Theorem~\ref{thm:boundary} identifies the admissible $a$-interval with $(0,A(b)]$, which proves \eqref{eq:A-variational} and \eqref{eq:Psi-negative}.
For $a<A(b)$, strict positivity of the density follows from Theorem~\ref{thm:boundary}.
In addition, $D_{a,-b}(0)=a/2-b>0$, while Lemma~\ref{lem:tails} implies eventual positivity and, after multiplication by $e^{t-a}$, growth of $D_{a,-b}(t)$ to $+\infty$ as $t\to\infty$.
Hence the positive continuous function $D_{a,-b}$ has a positive global minimum, proving the first assertion in \eqref{eq:Psi-sign}.
At $a=A(b)$, Corollary~\ref{cor:contact} gives global nonnegativity and a positive-argument zero, so the minimum equals zero and is attained.
\end{proof}

Proposition~\ref{prop:variational-A} leads to a nested one-dimensional computation.
The outer variable is $a$ and the inner problem is the global minimization of $D_{a,-b}$ in $t$.
Theorem~\ref{thm:boundary} supplies the initial bracket
\begin{equation}\label{eq:A-bracket}
a_L=\astar(1+b)<A(b)<a_R:=\min\{\astar,U(b)\},
\qquad -1<b<0.
\end{equation}
The left endpoint itself is admissible by Theorem~\ref{thm:linear}; Proposition~\ref{prop:extension} shows that it is strictly below the boundary.

\paragraph{Bracketed boundary algorithm.}
Fix $b\in(-1,0)$ and a tolerance $\varepsilon>0$.
Starting from \eqref{eq:A-bracket}, repeat the following steps.
\begin{enumerate}
\item Set $a_M=(a_L+a_R)/2$.
\item Compute, or rigorously enclose, the global minimum
\[
M=\inf_{t\ge0}D_{a_M,-b}(t).
\]
\item If $M\ge0$, replace $a_L$ by $a_M$; if $M<0$, replace $a_R$ by $a_M$.
\item Stop when $a_R-a_L\le\varepsilon$.
\end{enumerate}
Then $A(b)\in[a_L,a_R]$ throughout the iteration, and the midpoint gives an approximation with absolute error at most $\varepsilon/2$.
The no-gap result in Proposition~\ref{prop:coordinate} is what makes bisection in $a$ legitimate: the sign test cannot return to global nonnegativity after the boundary has been crossed.

For a fully explicit certified computation one may replace the theoretical bracket \eqref{eq:A-bracket} by
\begin{equation}\label{eq:A-certified-bracket}
a_L=2(1+b),\qquad
a_R=\min\left\{\dfrac73,U(b)\right\},
\end{equation}
using $2<\astar<7/3$. This avoids treating the published decimal approximation of $\astar$ as a certified numerical input.

A floating-point implementation of the bisection procedure is only a numerical approximation and is not by itself a proof of the sign of $M$.
The quantitative form of Lemma~\ref{lem:tails}, especially \eqref{eq:tail-bound-K}, gives a parameter-uniform remainder estimate on every compact parameter box $K$ contained in the interior domain.
If $L_K$ denotes the positive minimum of the leading coefficient and $C_K$ is any verified upper bound for the constant in \eqref{eq:tail-bound-K}, then any concrete $T\ge T_K$ satisfying
\[
\dfrac{C_K(1+\log T)}{T}\le\dfrac{L_K}{2}
\]
certifies
\begin{equation}\label{eq:validated-tail}
D_{a,c}(t)>0\qquad(t\ge T)
\end{equation}
throughout that box, because $D_{a,c}(t)=e^{t-a}\Phi_{a,c}(t)$.
Thus the present result provides a \emph{rigorous framework for validated numerical enclosure}; a fully certified implementation need only supply numerical interval enclosures for the finite constants entering Lemma~\ref{lem:tails} (or equivalent direct interval bounds).
Once such a $T$ is fixed, rigorous interval quadrature on $[0,T]$, together with interval evaluation of the $t$-derivatives in \eqref{eq:D-derivatives}, can certify the sign of the global minimum and hence rigorously update the bisection bracket.

\subsection{Tangency equations and conditional continuation}\label{subsec:continuation}
At a boundary point, every contact satisfies
\begin{equation}\label{eq:tangency-system}
D_{a,c}(t)=0,\qquad \partial_tD_{a,c}(t)=0,
\qquad a=A(-c),\quad 0<c<1.
\end{equation}
Consequently, solving the two equations in \eqref{eq:tangency-system} for $(a,t)$ at fixed $c$ is a useful fast predictor for the boundary.
As emphasized in Remark~\ref{rem:implicit}, however, the system alone is not a certificate: after a candidate $(a,t)$ is found one must still verify
\[
D_{a,c}(s)\ge0\qquad\text{for every }s\ge0.
\]
The bracketed procedure above provides such a global check and can therefore be combined with Newton or quasi-Newton iterations for speed.

The present results do not prove that the boundary contact is unique or always of exact order two.
The next local continuation statement therefore assumes both uniqueness and nondegeneracy; these hypotheses prevent a competing contact branch from becoming the active global boundary.

\begin{proposition}[Local continuation at a unique nondegenerate contact]\label{prop:continuation}
Let $b_0\in(-1,0)$, put $c_0=-b_0$ and $a_0=A(b_0)$, and suppose that the boundary density $D_{a_0,c_0}$ has a unique positive zero $t_0>0$.
Assume moreover that
\begin{equation}\label{eq:nondegenerate-contact}
D_{a_0,c_0}(t_0)=0,\qquad
D_t(a_0,c_0,t_0)=0,\qquad
D_{tt}(a_0,c_0,t_0)>0.
\end{equation}
Then there exists a neighborhood $I\subset(-1,0)$ of $b_0$ such that the optimal boundary $A$ and the corresponding contact point $t$ are real analytic on $I$.
More precisely, there exists a real-analytic function $t:I\to(0,\infty)$ such that
\[
D_{A(b),-b}(t(b))=0,
\qquad
D_t(A(b),-b,t(b))=0
\]
for every $b\in I$, with $t(b_0)=t_0$, and
\begin{equation}\label{eq:A-prime}
A'(b)=
\dfrac{D_c}{D_a}
\bigg|_{(a,c,t)=(A(b),-b,t(b))}>0.
\end{equation}
Moreover,
\begin{equation}\label{eq:t-prime}
t'(b)=
\dfrac{D_{tc}-D_{ta}A'(b)}{D_{tt}}
\bigg|_{(a,c,t)=(A(b),-b,t(b))}.
\end{equation}
\end{proposition}
\begin{proof}
Consider the real-analytic map
\[
\mathcal F(a,c,t)=\bigl(D_{a,c}(t),D_t(a,c,t)\bigr).
\]
Analyticity in the parameters follows from the locally convergent representation in Lemma~\ref{lem:convolution}, and analyticity in $t$ follows from Proposition~\ref{prop:density}.
At $(a_0,c_0,t_0)$ the Jacobian of $\mathcal F$ with respect to $(a,t)$ is
\[
\begin{pmatrix}
D_a & D_t\\
D_{ta} & D_{tt}
\end{pmatrix}.
\]
Since $D_t(a_0,c_0,t_0)=0$, its determinant is $D_aD_{tt}$.
By Corollary~\ref{cor:contact}, $D_a(a_0,c_0,t_0)<0$, while \eqref{eq:nondegenerate-contact} gives $D_{tt}(a_0,c_0,t_0)>0$.
Hence the analytic implicit-function theorem yields real-analytic functions
\[
a=\widetilde T(c),\qquad t=\widetilde t(c)
\]
for $c$ near $c_0$, satisfying
\[
D_{\widetilde T(c),c}(\widetilde t(c))=0,
\qquad
D_t(\widetilde T(c),c,\widetilde t(c))=0,
\]
and passing through $(a_0,t_0)$.

It remains to identify this local contact branch with the global optimal boundary.
Since $t_0$ is the unique positive zero of the nonnegative function $D_{a_0,c_0}$, choose an open interval $J$ containing $t_0$ so small that $D_{tt}>0$ in a neighborhood of $(a_0,c_0,t_0)$.
By positivity at the origin, the uniqueness of the zero, compactness on bounded intervals, and the uniform positive-tail estimate of Lemma~\ref{lem:tails}, there exist $T>t_0$ and $\eta>0$ such that $D_{a_0,c_0}\ge\eta$ on the compact part of $[0,T]\setminus J$, while the tail is positive for $t\ge T$.
After shrinking the parameter neighborhood if necessary, continuity on the compact part and the uniform tail estimate imply that
\[
D_{\widetilde T(c),c}(t)>0
\]
outside $J$ for all nearby $c$.
Inside $J$, shrink both $J$ and the parameter neighborhood once more, if necessary, so that
\[
D_{tt}(\widetilde T(c),c,t)>0
\qquad\text{for every }t\in J
\]
for all nearby $c$.
Hence $D_t(\widetilde T(c),c,\cdot)$ is strictly increasing on $J$ and vanishes at $t=\widetilde t(c)$.
Therefore $\widetilde t(c)$ is the unique minimizer on $J$, and the passage from the local critical point to positivity on the whole contact interval is explicit:
\[
D_{tt}>0\text{ on }J
\quad\Longrightarrow\quad
D_{\widetilde T(c),c}(t)
\ge D_{\widetilde T(c),c}(\widetilde t(c))=0,
\qquad t\in J.
\]
Together with the strict positivity outside $J$ proved above, this gives $D_{\widetilde T(c),c}\ge0$ on all of $[0,\infty)$, with a positive-argument zero.
Corollary~\ref{cor:contact} therefore gives
\[
\widetilde T(c)=T(c).
\]
Hence $A(b)=T(-b)$ and $t(b)=\widetilde t(-b)$ are real analytic near $b_0$.

Differentiating $D_{T(c),c}(t(c))=0$ and using $D_t=0$ gives
\[
D_aT'(c)+D_c=0,
\qquad
T'(c)=-\dfrac{D_c}{D_a}.
\]
Since $A(b)=T(-b)$, formula \eqref{eq:A-prime} follows.
At a boundary contact both $D_a$ and $D_c$ are negative by \eqref{eq:D-param-contact}, hence $A'(b)>0$.
Finally, differentiating $D_t(A(b),-b,t(b))=0$ with respect to $b$ gives
\[
D_{ta}A'(b)-D_{tc}+D_{tt}t'(b)=0,
\]
which is \eqref{eq:t-prime}.
\end{proof}

Proposition~\ref{prop:continuation} therefore applies only at a \emph{unique nondegenerate} boundary contact.
The global nonnegativity check remains essential: the local tangency equations alone do not identify the optimal boundary.

\subsection{The complete admissible interval for the second parameter}
\begin{corollary}[The inverse boundary]\label{cor:inverse}
The function $A$ has a continuous, strictly increasing inverse
\[
\beta:(0,\astar]\longrightarrow(-1,0],\qquad A(\beta(a))=a.
\]
For all $a>0$ and $b\in\R$,
\begin{equation}\label{eq:inverse}
F_{a,b}\in\BF\quad\Longleftrightarrow\quad0<a\le\astar
\quad\text{and}\quad\beta(a)\le b\le0.
\end{equation}
Furthermore,
\begin{equation}\label{eq:beta-endpoints}
\beta(\astar)=0,\qquad\lim_{a\downarrow0}\beta(a)=-1,
\end{equation}
and, for $0<a<\astar$,
\begin{equation}\label{eq:beta-bounds}
\max\left\{a-K_0,\;\dfrac{a-1-\sqrt{1+2a}}2\right\}
<\beta(a)<\dfrac a\astar-1.
\end{equation}
In particular, $a\le\astar$ is necessary for every admissible pair, and $a=\astar$ is possible only when $b=0$.
\end{corollary}
\begin{proof}
The endpoint limits and strict monotonicity in Theorem~\ref{thm:boundary} show that the range of $A$ is precisely $(0,\astar]$.
The inverse exists, is continuous and strictly increasing, and \eqref{eq:boundary} is equivalent to \eqref{eq:inverse}.
The endpoint assertions follow immediately.
The pair $(a,\beta(a))$ is admissible, so the strict necessary bounds give
\[
\beta(a)>a-K_0,\qquad
 a<2(1+\beta(a))+2\sqrt{1+\beta(a)}.
\]
Solving the latter inequality for $\beta(a)$ gives the second lower bound in \eqref{eq:beta-bounds}.
For the upper bound set $b_0=a/\astar-1\in(-1,0)$.
By \eqref{eq:A-bounds}, $A(b_0)>\astar(1+b_0)=a$.
Since $A$ is strictly increasing, $\beta(a)<b_0$.
\end{proof}

As a separate existential statement, the complete range of $b$ is
\begin{equation}\label{eq:existential}
\bigl(\exists a>0:F_{a,b}\in\BF\bigr)
\quad\Longleftrightarrow\quad-1<b\le0.
\end{equation}
This statement should not be confused with the fixed-$a$ interval in \eqref{eq:inverse}.

\section{Bounds and examples}\label{sec:comparisons}
\subsection{Necessary bounds, a sufficient line, and an exact implicit graph}
For $-1<b\le0$, the admissible set is
\[
\mathcal A_b=\{a>0:F_{a,b}\in\BF\}=(0,A(b)].
\]
The roles of the three descriptions are different.
The inequalities $a<U(b)$ are explicit necessary tests.
The inequality $a\le\astar(1+b)$ is an explicit sufficient condition in terms of the exact one-parameter constant.
The condition $a\le A(b)$ is necessary and sufficient, with $A$ determined implicitly by Corollary~\ref{cor:contact}.
Thus, for $-1<b<0$,
\begin{equation}\label{eq:inclusions}
(0,\astar(1+b)]\subsetneq(0,A(b)]
\subset(0,\min\{\astar,U(b)\}).
\end{equation}
At $b=0$, the precise identity is $\mathcal A_0=(0,\astar]$.
The strict upper bound by $\astar$ in \eqref{eq:inclusions} is specific to negative $b$.

Table~\ref{tab:bounds} compares the sufficient levels $2(1+b)$ and $\astar(1+b)$ with the explicit necessary envelope $U(b)$.
The first sufficient level follows from Theorem~\ref{thm:linear} and $2<\astar$; it is used here as an internal baseline, without an attribution of historical priority.
The $\astar$ column uses the published approximation \eqref{eq:astar-decimal-intro} from \cite{BMP2021}; the table does not report computed values of $A(b)$.
The additional necessary restriction $a\le\astar$, strict for $b<0$, can improve the last column.

\begin{table}[htbp]
\centering
\caption{Sufficient levels and a necessary envelope. Decimal entries are illustrative, not certified boundary values.}\label{tab:bounds}
\renewcommand{\arraystretch}{1.5}
\setlength{\tabcolsep}{10pt}
\begin{tabular}{@{}rrrr@{}}
\toprule
$b$ & $2(1+b)$ & $\astar(1+b)$, approx. & $U(b)$, approx.\\
\midrule
$0$ & $2$ & $2.2996564433$ & $2.9528280065$\\
$-\dfrac14$ & $1.5$ & $1.7247423324$ & $2.7028280065$\\
$-\dfrac12$ & $1$ & $1.1498282216$ & $2.4142135624$\\
$-\dfrac34$ & $0.5$ & $0.5749141108$ & $1.5$\\
$-\dfrac9{10}$ & $0.2$ & $0.2299656443$ & $0.8324555320$\\
\bottomrule
\end{tabular}
\end{table}

For example, at $b=-1/2$ the exact threshold satisfies
\[
\dfrac\astar2<A(-1/2)<\astar,
\]
whereas at $b=-3/4$ the bounds give
\[
\dfrac\astar4<A(-3/4)<\dfrac32.
\]
In particular, the limit $A(b)\to0$ as $b\downarrow-1$ is quantitative through $U(b)$.

\subsection{An explicit positive density on the comparison line}
The function
\[
x\longmapsto\left(1+\dfrac\astar{2x}\right)^{x-1/2}
\]
is Bernstein.
At this point on the sufficient comparison line, $\tau=1$, and \eqref{eq:rho-lower} gives
\[
\rho_{\astar/2,-1/2}(t)\ge\dfrac{5\astar}{36}
\exp\left(\dfrac\astar2-\dfrac{\astar t}2\right)>0.
\]
This point is not on the optimal boundary: its density is everywhere positive and the admissible interval extends to $A(-1/2)>\astar/2$.
The relative increase from coefficient $2$ to coefficient $\astar$ is $(\astar/2-1)\times100\%$, approximately $14.9828\%$.

\subsection{A fixed first parameter}
For $a=1$, Corollary~\ref{cor:inverse} gives
\[
F_{1,b}\in\BF\quad\Longleftrightarrow\quad\beta(1)\le b\le0,
\]
where
\begin{equation}\label{eq:beta1}
-\dfrac{\sqrt3}2<\beta(1)<\dfrac1\astar-1.
\end{equation}
The bounds are approximately $-0.8660254038$ and $-0.5651524370$.
They enclose the threshold $\beta(1)$, which is not evaluated here; they do not assert admissibility for every $b$ above the first number.
To locate the exact endpoint implicitly, one must impose global nonnegativity of $D_{1,-b}$ and contact with zero at a positive argument.

\subsection{Scope of the boundary description}
The no-gap theorem and strict coordinatewise ordering are the additional ingredients that turn a supremum into a genuine optimal boundary.
The exact Bernstein region can be described either by $a\le A(b)$ or by $\beta(a)\le b\le0$.
The contact criterion is valid without assuming that the zero is unique or nondegenerate.
Solving only two tangency equations, or checking finitely many sample values of a density, does not certify global nonnegativity.
An elementary formula for $A$, a certified numerical evaluation for general negative $b$, and a universal uniqueness theorem for the contact points are not claimed.

\section{The boundary-contact conjecture and further problems}\label{sec:conjecture}
Theorem~\ref{thm:boundary} proves that for each $-1<b<0$ the boundary density
\[
D_{A(b),-b}(t)
\]
is nonnegative on $[0,\infty)$ and has a nonempty finite set of positive zeros.
Corollary~\ref{cor:contact} shows that every such zero has finite even order at least two.
What is not yet determined is whether more than one contact can occur or whether a boundary zero can have order four or higher.
The numerical formulation in Subsection~\ref{subsec:numerical-boundary} and the local continuation result in Proposition~\ref{prop:continuation} suggest the following sharper picture.

\begin{conjecture}[Boundary-contact uniqueness]\label{conj:unique-contact}
For every $-1<b<0$ there exists a unique $t_b>0$ such that
\begin{equation}\label{eq:conjecture-contact}
D_{A(b),-b}(t_b)=0.
\end{equation}
Moreover, this zero is exactly of order two; equivalently,
\begin{equation}\label{eq:conjecture-double}
D_t(A(b),-b,t_b)=0,
\qquad
D_{tt}(A(b),-b,t_b)>0.
\end{equation}
\end{conjecture}

The conjecture is deliberately separated from the proved characterization.
Theorem~\ref{thm:boundary} and Corollary~\ref{cor:contact} require neither uniqueness nor nondegeneracy, so none of the unconditional results of the paper depend on Conjecture~\ref{conj:unique-contact}.

If Conjecture~\ref{conj:unique-contact} holds, its uniqueness and quadratic-contact assertions supply precisely the hypotheses of Proposition~\ref{prop:continuation} at every interior boundary point.
Consequently $A$ and $b\mapsto t_b$ are real analytic on $(-1,0)$ and satisfy
\begin{equation}\label{eq:conjectural-Aprime}
A'(b)=
\dfrac{D_c}{D_a}
\bigg|_{(a,c,t)=(A(b),-b,t_b)}>0,
\end{equation}
and
\begin{equation}\label{eq:conjectural-tprime}
t_b'(b)=
\dfrac{D_{tc}-D_{ta}A'(b)}{D_{tt}}
\bigg|_{(a,c,t)=(A(b),-b,t_b)}.
\end{equation}
Thus the conjecture would convert the presently implicit continuous boundary into a single analytic critical branch governed by the two tangency equations together with the global positivity condition.

Several related questions remain natural.
\begin{enumerate}
\item Prove or disprove Conjecture~\ref{conj:unique-contact} analytically, without relying on a finite numerical sampling of the density.
\item Determine the endpoint asymptotics of $A(b)$ and $t_b$ as $b\uparrow0$ and as $b\downarrow-1$.
In particular, it would be useful to identify the first nontrivial correction to the sufficient line $\astar(1+b)$ near either endpoint.
\item Develop certified high-precision enclosures of $A(b)$ on a mesh of $b$-values by combining the variational formula \eqref{eq:A-variational}, the uniform tail bounds of Lemma~\ref{lem:tails}, and interval arithmetic.
Such data could test the uniqueness conjecture and suggest asymptotic formulas, but would remain logically separate from a proof.
\item Determine whether the derivative formula \eqref{eq:conjectural-Aprime} admits effective upper and lower bounds depending only on $b$, and whether these lead to sharper explicit two-sided estimates for $A(b)$ than those in Theorem~\ref{thm:boundary}.
\end{enumerate}

The central open issue is therefore no longer the existence of an optimal boundary, which has been established above, but the fine geometry of its critical contact with the Laplace density.
A proof of Conjecture~\ref{conj:unique-contact} would provide a natural next step from the exact implicit solution toward a differential and computationally explicit description of the Alzer--Berg boundary.


\begin{thebibliography}{99}
\small
\bibitem{AlzerBerg2002}
H. Alzer and C. Berg,
Some classes of completely monotonic functions,
\textit{Ann. Acad. Sci. Fenn. Math.} \textbf{27} (2002), no.~2, 445--460.

\bibitem{AlzerBerg2006}
H. Alzer and C. Berg,
Some classes of completely monotonic functions, II,
\textit{Ramanujan J.} \textbf{11} (2006), no.~2, 225--248.
\doi{10.1007/s11139-006-6510-5}.

\bibitem{Berg2005}
C. Berg,
Problem 1. Bernstein functions,
\textit{J. Comput. Appl. Math.} \textbf{178} (2005), nos.~1--2, 525--526.
\doi{10.1016/j.cam.2004.09.003}.

\bibitem{Berg2008}
C. Berg,
Stieltjes--Pick--Bernstein--Schoenberg and their connection to complete monotonicity,
in \textit{Positive Definite Functions: From Schoenberg to Space-Time Challenges},
J. Mateu and E. Porcu (eds.), Department of Mathematics, University Jaume~I,
Castell\'on de la Plana, 2008, pp.~15--45.

\bibitem{BergForst1975}
C. Berg and G. Forst,
\textit{Potential Theory on Locally Compact Abelian Groups},
Ergebnisse der Mathematik und ihrer Grenzgebiete, vol.~87,
Springer-Verlag, Berlin--Heidelberg--New York, 1975.

\bibitem{BMP2021}
C. Berg, E. Massa and A. P. Peron,
A family of entire functions connecting the Bessel function $J_1$ and the Lambert $W$ function,
\textit{Constr. Approx.} \textbf{53} (2021), no.~1, 121--154.
\doi{10.1007/s00365-020-09499-x}.

\bibitem{BergPedersen2023}
C. Berg and H. L. Pedersen,
A family of Horn--Bernstein functions,
\textit{Exp. Math.} \textbf{32} (2023), no.~3, 505--513.
\doi{10.1080/10586458.2021.1980460}.

\bibitem{Bernstein1929}
S. Bernstein,
Sur les fonctions absolument monotones,
\textit{Acta Math.} \textbf{52} (1929), 1--66.
\doi{10.1007/BF02592679}.

\bibitem{CGDQ2023}
J. Cao, B.-N. Guo, W.-S. Du and F. Qi,
A sufficient and necessary condition for the power-exponential function
$(1+1/x)^{\alpha x}$ to be a Bernstein function and related $n$th derivatives,
\textit{Fractal Fract.} \textbf{7} (2023), no.~5, Article~397, 15~pp.
\doi{10.3390/fractalfract7050397}.

\bibitem{SSV2012}
R. L. Schilling, R. Song and Z. Vondra\v cek,
\textit{Bernstein Functions: Theory and Applications},
second revised and extended edition, De Gruyter Studies in Mathematics, vol.~37,
Walter de Gruyter, Berlin--Boston, 2012.
\doi{10.1515/9783110269338}.

\bibitem{SKJ2010}
E. Shemyakova, S. I. Khashin and D. J. Jeffrey,
A conjecture concerning a completely monotonic function,
\textit{Comput. Math. Appl.} \textbf{60} (2010), no.~5, 1360--1363.
\doi{10.1016/j.camwa.2010.06.017}.

\bibitem{Widder1941}
D. V. Widder,
\textit{The Laplace Transform},
Princeton Mathematical Series, vol.~6,
Princeton University Press, Princeton, NJ, 1941.
\end{thebibliography}
\end{document}